\documentclass[onefignum,onetabnum]{siamart171218}

\usepackage{lipsum}
\usepackage{amsfonts}
\usepackage{graphicx}
\usepackage{epstopdf}
\usepackage{algorithmic}
\usepackage{xcolor}
\ifpdf
  \DeclareGraphicsExtensions{.eps,.pdf,.png,.jpg}
\else
  \DeclareGraphicsExtensions{.eps}
\fi

\newsiamremark{remark}{Remark}
\newsiamremark{hypothesis}{Hypothesis}
\crefname{hypothesis}{Hypothesis}{Hypotheses}
\newsiamthm{claim}{Claim}

\newsiamremark{exam}{Example}

\def\<{\langle}
\def\>{\rangle}

\def\limn{\lim_{n\to\infty}}
\def\limsupn{\limsup_{n\to\infty}}
\def\liminfn{\liminf_{n\to\infty}}
\def\en{\enskip}

\title{Finite element analysis for identifying the  reaction coefficient in PDE from boundary observations
}

\author
{
Tran Nhan Tam Quyen\thanks{Institute for Numerical and Applied Mathematics, University of Goettingen,
Lotzestr. 16-18, 37083 Goettingen, Germany
  (\email{quyen.tran@uni-goettingen.de}).} 
}

\usepackage{amsopn}

\ifpdf
\hypersetup{
  pdftitle={Finite element analysis for identifying the  reaction coefficient in PDEs from boundary observations},
  pdfauthor={Tran Nhan Tam Quyen}
}
\fi

\begin{document}

\maketitle

% REQUIRED
\begin{abstract}
  This work is devoted to the nonlinear inverse problem of identifying the reaction coefficient in an elliptic boundary value problem from single Cauchy data on a part of the boundary. We then examine simultaneously two elliptic boundary value problems generated from the available Cauchy data. The output least squares method with the Tikhonov regularization is applied to find approximations of the sought coefficient. We discretize the PDEs with piecewise linear finite elements. The stability and convergence of this technique are then established. A numerical experiment is presented to illustrate our theoretical findings
\end{abstract}

% REQUIRED
\begin{keywords}
Reaction coefficient, finite element method, Tikhonov regularization, Neumann problem, mixed problem, ill-posed problem
\end{keywords}

% REQUIRED
\begin{AMS}
  35R25, 47A52, 35R30, 65J20, 65J22
\end{AMS}

\section{Introduction}
Let $\Omega$ be an open, bounded and connected domain of $\mathbb{R}^d,
~d\ge2$ with the boundary 
%polyhedral 
$\partial \Omega$ and $\Gamma\subset\partial\Omega$ be an
accessible part of the boundary which is relatively open. In this paper we are related with the following elliptic system
\begin{align}
-\nabla \cdot \big(\boldsymbol{\alpha} \nabla \Phi \big) + \beta \Phi &= f \quad\enskip \mbox{in} \quad \Omega,  \label{17-5-16ct1}\\
\boldsymbol{\alpha} \nabla \Phi \cdot \vec{\boldsymbol{n}} +\sigma \Phi &= j^\dag \quad \mbox{~on} \quad \Gamma,  \label{17-5-16ct1*}\\
\boldsymbol{\alpha} \nabla \Phi \cdot \vec{\boldsymbol{n}} +\sigma \Phi &= j_0 \quad \mbox{~on} \quad \partial\Omega\setminus\Gamma, \label{17-5-16ct1**}\\
 \Phi &= g^\dag \quad \mbox{~on} \quad \Gamma, \label{17-5-16ct1***}
\end{align}
where $\vec{\boldsymbol{n}}$ is the unit outward normal on $\partial\Omega$, the boundary conditions $j^\dag\in H^{-1/2}(\Gamma) := {H^{1/2}(\Gamma})^*$, $j_0\in H^{-1/2}(\partial\Omega\setminus\Gamma)$,~ $g^\dag\in H^{1/2}(\Gamma)$, the source term $f\in H^{-1}(\Omega) := {H^1(\Omega)}^*$ and the functions $\boldsymbol{\alpha},~ \sigma$ are assumed to be known. Here, $\sigma \in L^\infty(\partial\Omega)$ with $\sigma(x)\ge 0$ a.e. on $\partial\Omega$ and $\boldsymbol{\alpha} :=  \left(\alpha_{rs}\right)_{1\le r, s\le d} \in {L^{\infty}(\Omega)}^{d \times d}$ is a symmetric diffusion matrix satisfying the uniformly elliptic condition
$$
\boldsymbol{\alpha}(x)\xi \cdot\xi = \sum_{1\le r,s\le d} \alpha_{rs}(x)\xi_r\xi_s \ge \underline{\alpha} |\xi|^2
$$
a.e. in $\Omega$ for all $\xi = \left(\xi_r\right)_{1\le r\le d} \in \mathbb{R}^d$ with some constant $\underline{\alpha} >0$. In case $\boldsymbol{\alpha} = \alpha\cdot \boldsymbol{I}_d$, the unit $d\times d$-matrix $\boldsymbol{I}_d$ and $\alpha: \Omega \to \mathbb{R}$, then $\alpha$ is called the scalar diffusion.

The system \eqref{17-5-16ct1}--\eqref{17-5-16ct1***} is overdetermined, i.e. on the part $\Gamma$ of the boundary $\partial\Omega$  the Neumann data and Dirichlet data are supplemented simultaneously. Therefore, if the reaction coefficient
\begin{align}\label{2-11-18ct1}
\beta \in \mathcal{S}_{ad} := \{\beta \in L^\infty(\Omega) ~|~ 0< \underline{\beta} \le \beta(x) \le \overline{\beta} \enskip \mbox{a.e. in} \enskip \Omega \}
\end{align}
is given also, there may be no $\Phi$ satisfying the system, where the constants $0< \underline{\beta} \le \overline{\beta}$ are known. In this paper we assume that the system is consistent and our aim is to reconstruct the coefficient $\beta\in \mathcal{S}_{ad}$ from several sets of observation data $\big(j_{\delta}^i,g_{\delta}^i\big)_{i=1,\ldots I} \subset H^{-1/2}(\Gamma) \times  H^{1/2}(\Gamma)$ of the exact $\big(j^\dag,g^\dag\big)$ obeying the deterministic noise model
\begin{align}\label{31-5-17ct1}
\frac{1}{I}\sum_{i=1}^I\left( \big\|j_{\delta}^i - j^\dag\big\|_{H^{-1/2}(\Gamma)} + \big\|g_{\delta}^i - g^\dag\big\|_{H^{1/2}(\Gamma)}\right) \le \delta
\end{align}
with $\delta >0$ denoting the error level of the observations.

%In this paper we investigate the stability and convergence of finite element methods for the problem of identifying the reaction coefficient in elliptic PDEs from observations on an accessible part of the boundary. This problem has not been studied so far, the existing results used distributed measurements which is very hard to apply in practice.

The problem arises from different contexts of applied sciences, e.g., from aquifer analysis, optical tomography which attracted great attention of many scientists in the
last 40 years or so. For surveys on the subject, we refer the reader to \cite{BaKu89,Sun94,Tar05}. Although there have been many papers devoted to the subject, the authors however used the {\it distributed} observations, i.e. the measurement data is assumed to be given in the whole domain $\Omega$, see, e.g., 
Alt
\cite{Alt91},
Colonius and Kunisch
\cite{CoKu86},
Engl {\it et al.}
\cite{EKN89},
Kaltenbacher and Hofmann 
\cite{KaHo10},  
Kaltenbacher and Klassen 
\cite{KaKl18}, 
Neubauer
\cite{Ne92},
Resmerita and Scherzer
\cite{ReSc06}
and 
\cite{haqu10,haqu11,haqu12,haqu12jmaa} and the references therein.
We mention that the {\it boundary} observation subject, i.e. the measurement data is available only on the part $\Gamma$ of the boundary $\partial\Omega$ as defined in the present paper for this identification problem, which is more realistic
from the practical point of view, has not yet been investigated so far.

For simplicity of exposition we below consider one observation pair $\left(j_\delta,g_\delta\right)$ being available, i.e. $I=1$, while the approach described here can be naturally extended to multiple measurements. We from the available observation data $\left(j_\delta,g_\delta\right)$ simultaneously examine the Neumann boundary value problem
\begin{align}
-\nabla \cdot \big(\boldsymbol{\alpha} \nabla u \big) + \beta u &= f \quad \mbox{~in} \quad \Omega,  \label{17-5-16ct2}\\
\boldsymbol{\alpha} \nabla u \cdot \vec{\boldsymbol{n}} +\sigma u &= j_\delta \quad \mbox{on} \quad \Gamma,  \label{17-5-16ct2*}\\
\boldsymbol{\alpha} \nabla u \cdot \vec{\boldsymbol{n}} +\sigma u &= j_0 \quad \mbox{on} \quad \partial\Omega\setminus\Gamma  \label{17-5-16ct2**}
\end{align}
and the mixed boundary value problem
\begin{align}
-\nabla \cdot \big(\boldsymbol{\alpha} \nabla v \big) + \beta v &= f \quad \mbox{~in} \quad \Omega,  \label{17-5-16ct3}\\
v &= g_\delta \quad \mbox{on} \quad \Gamma,  \label{17-5-16ct3*}\\
\boldsymbol{\alpha} \nabla v \cdot \vec{\boldsymbol{n}} +\sigma v &= j_0 \quad \mbox{on} \quad \partial\Omega\setminus\Gamma.  \label{17-5-16ct3**}
\end{align}
Let $N_{j_\delta}(\beta)$ and $M_{g_\delta}(\beta)$ be the unique weak solutions of \cref{17-5-16ct2}--\cref{17-5-16ct2**} and \cref{17-5-16ct3}--\cref{17-5-16ct3**}, respectively. We then consider a minimizer $\beta_{\delta,\rho}$ of the Tikhonov regularized minimization problem
$$
\min_{\beta \in \mathcal{S}_{ad}} J_{\delta,\rho}(\beta), \quad J_{\delta,\rho} (\beta) := \big\|N_{j_\delta}(\beta)- M_{g_\delta}(\beta)\big\|^2_{L^2(\Omega)} + \rho\|\beta-\beta^*\|^2_{L^2(\Omega)} \eqno \left(\mathcal{P}_{\delta,\rho}\right) 
$$
as reconstruction, where $\rho>0$ is the regularization parameter and $\beta^*$ is an a priori estimate of the true coefficient. The motivation for using the above cost functional is that $\big\|N_{j_\delta}(\beta)- M_{g_\delta}(\beta)\big\|_{L^2(\Omega)} \ge 0$ and at the sought coefficient $\beta$ it holds the identity $\big\|N_{j^\dag}(\beta)- M_{g^\dag}(\beta)\big\|_{L^2(\Omega)} = 0$.  

Let $N^h_{j_\delta}(\beta)$ and $M^h_{g_\delta}(\beta)$ be corresponding approximations of $N_{j_\delta}(\beta)$ and $M_{g_\delta}(\beta)$ in the finite dimensional space $\mathcal{V}^h_1$ of piecewise linear, continuous finite elements. Utilizing the variational discretization concept \cite{Hin05} of $\left(\mathcal{P}_{\delta,\rho}\right)$
that avoids explicit discretization of the control variable, we then consider the discrete regularized problem corresponding to $\left(\mathcal{P}_{\delta,\rho}\right)$, i.e. the following minimization problem
$$
\min_{\beta\in\mathcal{S}_{ad}} J^h_{\delta,\rho}(\beta), \quad  J^h_{\delta,\rho}(\beta) := \big\|N^h_{j_\delta}(\beta) - M^h_{g_\delta}(\beta) \big\|^2_{L^2(\Omega)} + \rho\|\beta - \beta^*\|^2_{L^2(\Omega)} \eqno \big(\mathcal{P}^h_{\delta,\rho}\big)
$$ 
which also attains a minimizer $\beta^h_{\delta,\rho}$ satisfying the relation (cf. \Cref{discrete1})
\begin{align*}
\beta^h_{\delta,\rho} (x)= \mathcal{P}_{[\underline{\beta}, \overline{\beta}]} \left( \frac{1}{\rho} \big(N^h_{j_\delta}(\beta^h_{\delta,\rho})(x)A^h_N(\beta^h_{\delta,\rho})(x) - M^h_{g_\delta}(\beta^h_{\delta,\rho}) (x)A^h_M(\beta^h_{\delta,\rho})(x)\big) + \beta^*(x)\right)
\end{align*}
a.e. in $\Omega$, where $\mathcal{P}_{[\underline{\beta}, \overline{\beta}]} (c) := \max\left( \underline{\beta}, \min\big(c, \overline{\beta}\big)\right)$, the states $A^h_N$ and $A^h_M$ are finite $\mathcal{V}^h_1$-element approximations of solutions to suitably chosen adjoint problems. This identity will be exploited to the gradient projection algorithm presented in \Cref{iterative}.

In \Cref{proof} we show that the proposed finite element method is stable, i.e. if the regularization parameter and the observation data are both fixed, then the sequence of minimizers $\big(\beta^h_{\delta,\rho}\big)_{h>0}$ to $\big(\mathcal{P}^h_{\delta,\rho}\big)$ can be extracted a subsequence which converges in the $L^2(\Omega)$-norm to a solution of $\big(\mathcal{P}_{\delta,\rho}\big)$ as the mesh size $h$ of the triangulation $\mathcal{T}^h$ tends to zero. Furthermore as $h,\delta \to 0$ and with an appropriate a priori regularization parameter choice $\rho=\rho(h,\delta) \to 0$,
the whole sequence $\big(\beta^h_{\delta,\rho}\big)_{\rho>0}$ converges in the $L^2(\Omega)$-norm to the $\beta^*$-minimum-norm solution $\beta^\dag$ of the
identification problem  defined by
$$\beta^\dag = \arg \min_{\left\{\beta \in \mathcal{S}_{ad} ~|~ N_{j^\dag}(\beta) = M_{g^\dag}(\beta) \right\}} \|\beta - \beta^*\|_{L^2(\Omega)}.$$
The corresponding state sequences $\big( N^{h}_{j_{\delta}}\big(\beta^h_{\delta,\rho}\big)\big)_{\rho>0}$ and $\big( M^{h}_{g_{\delta}}\big(\beta^h_{\delta,\rho}\big)\big)_{\rho>0}$ then converge in the $H^1(\Omega)$-norm to the exact state $\Phi^\dag = \Phi(j^\dag,g^\dag,\beta^\dag)$ of the problem \cref{17-5-16ct1}--\cref{17-5-16ct1***}. 

Our numerical implementation will be presented in  \Cref{iterative}. First, for the numerical
solution of the discrete regularized problem $\big(\mathcal{P}^h_{\delta,\rho}\big)$ we employ a gradient projection algorithm with Armijo steplength rule. In \Cref{exam1} we assume that observations are available on the bottom surface of the domain. \Cref{exam2} is a continuity of the first one, where we investigate the effect of the regularization parameter choice rule and previous iteration processes as well on the final computed numerical result. In case observations taking on the bottom and left surface the computation is given in \Cref{exam3}, while \Cref{exam4} is devoted to multiple measurements.

To complete this introduction we wish to mention briefly some parameter identification problems in PDEs from boundary observations. The authors Xie and Zou \cite{XiZo05}, Xu and Zou \cite{XuZo15} have used finite element methods to numerically recovered the fluxes on the inaccessible boundary $\Gamma_i$ from measurement data of the state on the accessible boundary $\Gamma_a$, while the problem of identifying the Robin coefficient on $\Gamma_i$ is also investigated by Xu and Zou \cite{XuZo15a}. Recently, authors of \cite{HHQ19,HKQ18} adopted the variational approach of Kohn and Vogelius combined with quadratic stabilizing penalty term and total variation regularization technique to the source term and scalar diffusion coefficient identification, respectively, using observations available on the whole boundary.

Throughout the paper the symbol $A \preceq B$ refers to the inequality $A \le cB$ for some constant $c$ independent of both $A$ and $B$. In the Lebesgue space $L^2(Q)$, where $Q$ is either $\Omega$, $\partial\Omega$ or $\Gamma$, we use for all $y, \widehat{y}\in L^2(Q)$ the inner product and the corresponding norm as
$(y,\widehat{y})_Q := \int_Q y(x)\cdot\widehat{y}(x) dx$ and $\|y\|_Q := (y,y)^{1/2}_Q.$
%If not stated otherwise we for short will write
%$\int_\Omega \cdots$ instead of $\int_\Omega \cdots dx$. 
We also use the standard notion of Sobolev spaces $H^k(Q) := W^k_2(Q)$ from, e.g., \cite{adams} with notations of its inner product $(\cdot,\cdot)_{k,Q}$, the norm $\|\cdot\|_{k,Q}$ and the semi-norm $|\cdot|_{k,Q}$. Note that $\|\cdot\|_{0,Q} = |\cdot|_{0,Q} =\|\cdot\|_Q$.

\section{Finite element discretization}\label{discrete1}

\subsection{Preliminaries}

We remark that the expression 
\begin{align*}
[u,v] := [u,v]_{(\boldsymbol{\alpha},\beta,\sigma)} := (\boldsymbol{\alpha} \nabla u, \nabla v)_\Omega + (\beta u,v)_\Omega + (\sigma u, v)_{\partial\Omega} 
\end{align*}
generates an inner product on the space
$ H^1(\Omega)$ which is equivalent to the usual one, i.e. there exist positive constants $c_1, c_2$ such that
\begin{align}\label{18-10-16ct1}
c_{1}\|u\|_{1, \Omega} \le [u,u]_{(\boldsymbol{\alpha},\beta,\sigma)} \le c_2\|u\|_{1,\Omega}
\end{align}
for all $u\in H^1(\Omega)$ and $\beta\in \mathcal{S}_{ad}$, where $c_1$ and $c_2$ are independent of $\beta$. Therefore, for each $\beta\in\mathcal{S}_{ad}$ the Neumann boundary value problem \cref{17-5-16ct2}--\cref{17-5-16ct2**}
defines a unique weak solution $u=u(\beta) := N_{j_\delta}(\beta)$ in the sense that $N_{j_\delta}(\beta) \in H^1(\Omega)$ and the equation
\begin{align}\label{17-10-16ct2}
\big[N_{j_\delta}(\beta),\phi\big]_{(\boldsymbol{\alpha},\beta,\sigma)}  = \<f,\phi\>_\Omega + \<j_\delta, \phi\>_{\Gamma} + \<j_0, \phi\>_{\partial\Omega\setminus\Gamma}  
\end{align}
is satisfied for all $\phi\in H^1(\Omega)$, where $\<\cdot,\cdot\>_\Omega$ and $\<\cdot, \cdot\>_{\Gamma}$ stand for the dual pairs $\<\cdot,\cdot\>_{\big(H^{-1}(\Omega), H^1(\Omega)\big)}$ and $\<\cdot, \cdot\>_{\big(H^{-1/2}(\Gamma), H^{1/2}(\Gamma)\big)}$, respectively. Furthermore, there holds the estimate
\begin{align}\label{17-10-16ct4}
\big\|N_{j_\delta}(\beta)\big\|_{1,\Omega} \preceq \|j_\delta\|_{H^{-1/2}(\Gamma)} + \|j_0\|_{H^{-1/2}(\partial\Omega\setminus\Gamma)} + \|f\|_{H^{-1}(\Omega)}.
\end{align}
A function $v =v(\beta) := M_{g_\delta}(\beta)$ is said to be a (unique) weak solution of the mixed boundary value problem \cref{17-5-16ct3}--\cref{17-5-16ct3**}
if $M_{g_\delta}(\beta)\in H^1(\Omega)$ with $M_{g_\delta}(\beta)_{|\Gamma} = g_\delta$ and the equation
\begin{align}\label{17-10-16ct2*}
\big[M_{g_\delta}(\beta),\phi\big]_{(\boldsymbol{\alpha},\beta,\sigma)}  = \<f,\phi\>_\Omega  + \<j_0, \phi\>_{\partial\Omega\setminus\Gamma} 
\end{align}
is satisfied for all $\phi\in H^1_0(\Omega\cup\Gamma)$, where 
$
H^1_0(\Omega\cup\Gamma) := \overline{C^\infty_c(\Omega\cup\Gamma)}^{H^1(\Omega)} = \{\phi \in H^1(\Omega) ~|~ \phi_{|\Gamma} = 0\}$,
the bar denotes the closure in $H^1(\Omega)$ and $C^\infty_c(\Omega\cup\Gamma)$ is the set of all functions $\phi\in C^\infty(\overline{\Omega})$ with $\mbox{supp}\phi$ being a compact subset of $\Omega\cup\Gamma$ (see, e.g., \cite[pp.\ 9, 67]{Troianiello}). The above weak solution satisfies  the estimate
\begin{align}\label{14-11-18ct3}
\|M_{g_\delta}(\beta)\|_{1,\Omega} \preceq \|g_\delta\|_{H^{1/2}(\Gamma)} + \|j_0\|_{H^{-1/2}(\partial\Omega\setminus\Gamma)} + \|f\|_{H^{-1}(\Omega)}.
\end{align}

\begin{remark}\label{27-12-18ct2}
We mention that under additional assumptions $\boldsymbol{\alpha} \in {W^{1,\infty}(\Omega)}^{d\times d}$, $j_\delta \in H^{1/2}(\Gamma)$, $j_0 \in H^{1/2}(\partial\Omega\setminus\Gamma)$, $g_\delta \in H^{3/2}(\Gamma)$, $f \in L^2(\Omega)$ and either $\partial\Omega$ is smooth of the class $C^{0,1}$ or the domain $\Omega$ is convex, the weak solutions $N_{j_\delta}(\beta),~ M_{g_\delta}(\beta)\in H^2(\Omega)$ for all $\beta \in \mathcal{S}_{ad}$ (see, e.g., \cite{Grisvad,Troianiello}) satisfying
\begin{align}\label{17-10-16ct4*}
\|N_{j_\delta}(\beta)\|_{2,\Omega} \preceq \|j_\delta\|_{H^{1/2}(\partial\Omega)} + \|j_0\|_{H^{1/2}(\partial\Omega\setminus\Gamma)} + \|f\|_{L^2(\Omega)}
\end{align}
and 
\begin{align}\label{17-10-16ct4**}
\|M_{g_\delta}(\beta)\|_{2,\Omega} \preceq \|g_\delta\|_{H^{3/2}(\partial\Omega)} + \|j_0\|_{H^{1/2}(\partial\Omega\setminus\Gamma)} + \|f\|_{L^2(\Omega)}.
\end{align}
\end{remark}

We now state some properties of the coefficient-to-solution operators 
$$N_{j_\delta},~ M_{g_\delta}: \mathcal{S}_{ad} \to H^1(\Omega).$$

\begin{lemma}\label{differential} Assume that the dimension $d\le 4$. Then the operators $N_{j_\delta}$ and $M_{g_\delta}$ are infinitely Fr\'echet differentiable on the set $\mathcal{S}_{ad}$ with respect to the $L^2(\Omega)$-norm. For $\beta\in\mathcal{S}_{ad}$ and $(\kappa_1, \ldots, \kappa_m) \in {L^\infty(\Omega)}^m$ the $m$-th order differentials $D_N^{(m)} := N_{j_\delta}^{(m)}(\beta)(\kappa_1, \ldots, \kappa_m) \in H^1(\Omega)$ and $D_M^{(m)} := M_{g_\delta}^{(m)}(\beta)(\kappa_1, \ldots, \kappa_m) \in H^1_0(\Omega\cup\Gamma)$ are the unique solutions to the variational equations
\begin{align*}
\big[D_N^{(m)},\phi\big]_{(\boldsymbol{\alpha},\beta,\sigma)} = -\sum_{i=1}^m \left(\kappa_i N_{j_\delta}^{(m-1)}(\beta) \xi_i,\phi\right)_\Omega, \en \forall \phi\in H^1(\Omega)
\end{align*}
and
\begin{align*}
\big[D_M^{(m)},\phi\big]_{(\boldsymbol{\alpha},\beta,\sigma)} = -\sum_{i=1}^m \left(\kappa_i M_{g_\delta}^{(m-1)}(\beta) \xi_i,\phi\right)_\Omega, \en \forall \phi\in H^1_0(\Omega\cup\Gamma)
\end{align*}
with $\xi_i := (\kappa_1, \ldots, \kappa_{i-1},\kappa_{i+1}, \ldots, \kappa_m) \in {L^\infty(\Omega)}^{m-1}$, respectively. Furthermore, 
\begin{align*}
\max\left(\big\|D_N^{(m)}\big\|_{H^1(\Omega)}, \big\|D_M^{(m)}\big\|_{H^1(\Omega)} \right) \preceq \prod_{i=1}^m \|\kappa_i\|_{L^2(\Omega)}.
\end{align*}
\end{lemma}

\begin{proof}
The proof is based on standard arguments, therefore omitted here.
\end{proof}

We mention that the restriction on the dimension $d\le 4$ in the above \Cref{differential} is removed if the $L^\infty(\Omega)$-norm is taken into account instead of the $L^2(\Omega)$-norm (see, e.g., \cite{haqu10,haqu12}).

\begin{lemma}\label{weakly conv.}
Assume that the sequence $\left( \beta_n\right)_n\subset \mathcal{S}_{ad}$
converges weakly in $L^2(\Omega)$ to an element $\beta$. Then the sequences $\left(N_{j_\delta}(\beta_n)\right)_n$ and $\left(M_{g_\delta}(\beta_n)\right)_n$ converge respectively to $N_{j_\delta}(\beta)$ and $M_{g_\delta}(\beta)$ weakly in $H^1(\Omega)$ and strongly in the $L^2(\Omega)$-norm.
\end{lemma}

\begin{proof}
We first note that since $\mathcal{S}_{ad}$ is a convex and closed subset of $L^2(\Omega)$, it is weakly closed in $L^2(\Omega)$ which implies that $\beta\in \mathcal{S}_{ad}$. Furthermore, it is a weakly$^*$ compact subset of $L^\infty(\Omega)$ (see, e.g., \cite[Remark 2.1]{quyen18}). 
Therefore, by the inequality \cref{17-10-16ct4} and the embeddings $H^1(\Omega)\hookrightarrow L^2(\Omega)$ as well as $H^1(\Omega)\hookrightarrow L^2(\partial\Omega)$ being compact (see, e.g., \cite{Lady84,Mc10}), the sequences $(N_n)_n :\equiv \big(N_{j_\delta}(\beta_n)\big)_n$ and $\left( \beta_n\right)_n$ have subsequences denoted by the same symbol such that
\begin{align}
&\beta_n \rightharpoonup \beta ~\mbox{weakly}^* ~\mbox{in}~ L^\infty(\Omega), \en \mbox{i.e.} \en (\beta_n,\xi)_\Omega \to (\beta,\xi)_\Omega \en \mbox{for all} \en \xi \in L^1(\Omega), \label{14-5-19ct1}\\
&N_n \rightharpoonup \theta_N ~\mbox{weakly in}~ H^1(\Omega) ~\mbox{and strongly in}~ L^2(\Omega) ~\mbox{and}~ L^2(\partial\Omega) \label{14-5-19ct2}
\end{align}
as $n\to\infty$, where $\theta_N$ is an element of $H^1(\Omega)$. We can show that $\theta_N = N_{j_\delta}(\beta)$. In fact, for all $\phi\in H^1(\Omega)$, we have from \eqref{17-10-16ct2} for all $n\in \mathbb{N}$ that
\begin{align*}
&\<f,\phi\>_\Omega + \<j_\delta, \phi\>_{\Gamma} + \<j_0, \phi\>_{\partial\Omega\setminus\Gamma}\notag\\
&~\quad = (\boldsymbol{\alpha} \nabla N_n, \nabla \phi)_\Omega + (\beta_n N_n,\phi)_\Omega + (\sigma N_n, \phi)_{\partial\Omega} \notag\\
&~\quad = (\boldsymbol{\alpha} \nabla \theta_N, \nabla \phi)_\Omega + (\beta \theta_N,\phi)_\Omega + (\sigma \theta_N, \phi)_{\partial\Omega} \notag\\
&~\quad \quad + (\boldsymbol{\alpha}\nabla (N_n - \theta_N),\nabla \phi)_\Omega + (\beta_n - \beta, \theta_N\phi)_\Omega \\
&~\quad \quad + (\beta_n(N_n - \theta_N),\phi)_\Omega + (\sigma(N_n-\theta_N),\phi)_{\partial\Omega}.
\end{align*}
By \eqref{14-5-19ct2}, we get $(\boldsymbol{\alpha}\nabla (N_n - \theta_N),\nabla \phi)_\Omega \to 0$ as $n \to \infty$. Furthermore, we have
\begin{align*}
&\left|(\beta_n(N_n - \theta_N),\phi)_\Omega + (\sigma(N_n-\theta_N),\phi)_{\partial\Omega}\right|\\
&~\quad \le \overline{\beta} \|\phi\|_\Omega \|N_n - \theta_N\|_\Omega
+\|\sigma\|_{L^\infty(\partial\Omega)} \|\phi\|_{\partial\Omega} \|N_n - \theta_N\|_{\partial\Omega} \to 0
\end{align*}
as $n \to \infty$, here we used \eqref{14-5-19ct2} again. Since $\theta_N\phi\in L^1(\Omega)$, it follows from \eqref{14-5-19ct1} that $(\beta_n - \beta, \theta_N\phi)_\Omega \to 0$ as $n \to \infty$. Therefore, sending $n \to \infty$ in the above equation we arrive at 
$$\<f,\phi\>_\Omega + \<j_\delta, \phi\>_{\Gamma} + \<j_0, \phi\>_{\partial\Omega\setminus\Gamma}
 = (\boldsymbol{\alpha} \nabla \theta_N, \nabla \phi)_\Omega + (\beta \theta_N,\phi)_\Omega + (\sigma \theta_N, \phi)_{\partial\Omega}$$
 for all $\phi\in H^1(\Omega)$, that means $\theta_N = N_{j_\delta}(\beta)$. With similar arguments we also obtain that $\left(M_{g_\delta}(\beta_n)\right)_n$ converges to $M_{g_\delta}(\beta)$ weakly in $H^1(\Omega)$, which finishes the proof.
\end{proof}

Together with \cref{17-5-16ct2}--\cref{17-5-16ct3**}, we consider {\it two adjoint problems}
\begin{align}
-\nabla \cdot \big(\boldsymbol{\alpha} \nabla A_N \big) + \beta A_N 
&= N_{j_\delta}(\beta) - M_{g_\delta}(\beta) \en \mbox{in} \en \Omega \label{17-10-16ct1*}\\
\boldsymbol{\alpha} \nabla A_N \cdot \vec{\boldsymbol{n}} +\sigma A_N &= 0 \en \mbox{on} \en \partial\Omega\label{17-10-16ct1***}
\end{align}
and 
\begin{align}
-\nabla \cdot \big(\boldsymbol{\alpha} \nabla A_M \big) + \beta A_M 
&= N_{j_\delta}(\beta) - M_{g_\delta}(\beta) \en \mbox{in} \en \Omega \label{17-10-16ct1*+}\\
A_M &= 0 \en \mbox{on} \en \Gamma\label{17-10-16ct1**+}\\
\boldsymbol{\alpha} \nabla A_M \cdot \vec{\boldsymbol{n}} +\sigma A_M &= 0 \en \mbox{on} \en \partial\Omega \setminus \Gamma\label{17-10-16ct1***+}
\end{align}
that attain unique weak solutions $A_N = A_N(\beta)$ and $A_M = A_M(\beta)$ in the sense that $A_N(\beta) \in H^1(\Omega)$ and $A_M (\beta)\in H^1_0(\Omega\cup\Gamma)$ satisfy the variational equations
\begin{align}\label{7-11-18ct3}
[A_N(\beta),\phi]_{(\boldsymbol{\alpha},\beta,\sigma)} &= \big(N_{j_\delta}(\beta) - M_{g_\delta}(\beta),\phi\big)_\Omega, \en \forall \phi\in H^1(\Omega)\\
[A_M(\beta),\phi]_{(\boldsymbol{\alpha},\beta,\sigma)} &= \big(N_{j_\delta}(\beta) - M_{g_\delta}(\beta),\phi\big)_\Omega, \en \forall \phi\in H^1_0(\Omega\cup\Gamma).
\end{align}
Furthermore,
\begin{align*}
&\max\left(\|A_N(\beta)\|_{1,\Omega}, \|A_M(\beta)\|_{1,\Omega} \right)  \\
&~\quad \preceq \|j_\delta\|_{H^{-1/2}(\Gamma)} + \|g_\delta\|_{H^{1/2}(\Gamma)} + \|j_0\|_{H^{-1/2}(\partial\Omega\setminus\Gamma)} + \|f\|_{H^{-1}(\Omega)}.
\end{align*}

\begin{theorem}\label{existance}
The minimization problem 
$$
\min_{\beta \in \mathcal{S}_{ad}} J_{\delta,\rho}(\beta), \quad J_{\delta,\rho} (\beta) := \big\|N_{j_\delta}(\beta)- M_{g_\delta}(\beta)\big\|^2_\Omega + \rho\|\beta-\beta^*\|^2_\Omega \eqno \left(\mathcal{P}_{\delta,\rho}\right) 
$$
attains a minimizer $\beta_{\delta,\rho}$ which satisfies the identity
\begin{align*}
\beta_{\delta,\rho} (x)= \mathcal{P}_{[\underline{\beta}, \overline{\beta}]} \left( \frac{1}{\rho} \big(N_{j_\delta}(\beta_{\delta,\rho})(x)A_N(\beta_{\delta,\rho})(x) - M_{g_\delta}(\beta_{\delta,\rho}) (x) A_M(\beta_{\delta,\rho})(x)\big) + \beta^*(x)\right) 
\end{align*}
a.e. in $\Omega$, where $A_N,~ A_M$ come from \cref{17-10-16ct1*}--\cref{17-10-16ct1***+}.
\end{theorem}

\begin{proof}
The existence of a minimizer $\beta$ follows directly from \Cref{weakly conv.}. It remains to  show the above identity. Due to the first order optimality condition for the minimizer $\beta$, we get for all $\gamma \in \mathcal{S}_{ad}$ that $J'_{\delta,\rho}(\beta)(\gamma - \beta) \ge 0$. Setting $\kappa := \gamma-\beta$, we then obtain
\begin{align*}
\big(N_{j_\delta}(\beta)- M_{g_\delta}(\beta), N'_{j_\delta}(\beta) \kappa - M'_{g_\delta}(\beta)\kappa \big)_\Omega + \rho(\beta-\beta^*, \kappa)_\Omega \ge 0.
\end{align*}
It follows from \Cref{differential} that
\begin{align*}
&\big(N_{j_\delta}(\beta)- M_{g_\delta}(\beta), N'_{j_\delta}(\beta) \kappa - M'_{g_\delta}(\beta)\kappa \big)_\Omega\\ 
&~\quad = \big(N_{j_\delta}(\beta)- M_{g_\delta}(\beta), N'_{j_\delta}(\beta) \kappa \big)_\Omega - \big(N_{j_\delta}(\beta)- M_{g_\delta}(\beta), M'_{g_\delta}(\beta)\kappa\big)_\Omega\\
&~\quad = [A_N(\beta), N'_{j_\delta}(\beta)\kappa]_{(\boldsymbol{\alpha},\beta,\sigma)} - [A_M(\beta), M'_{g_\delta}(\beta)\kappa]_{(\boldsymbol{\alpha},\beta,\sigma)} \\
&~\quad = -(A_N(\beta) N_{j_\delta}(\beta),\kappa)_\Omega + (A_M(\beta) M_{g_\delta}(\beta), \kappa)_\Omega
\end{align*}
which yields
$$
\left( \frac{1}{\rho} \big(N_{j_\delta}(\beta)A_N(\beta) - M_{g_\delta}(\beta)A_M(\beta)\big) + \beta^* - \beta , \gamma - \beta \right)_\Omega \le 0
$$
for all $\gamma \in \mathcal{S}_{ad}$. This completes the proof.
\end{proof}

%Now, we are in the position to prove the main result of this section.

\subsection{Finite element discretization}\label{discrete}

Let $\left(\mathcal{T}^h\right)_{0<h<1}$ be a family of quasi-uniform triangulations of the domain $\overline{\Omega}$ with the mesh size $h$. 
For the definition of the discretization space of the state
functions let us denote
\begin{align*}
\mathcal{V}_k^h := \left\{v^h\in C(\overline\Omega)
~|~{v^h}_{|T} \in \mathcal{P}_k, ~~\forall
T\in \mathcal{T}^h\right\} \quad\mbox{and}\quad
\mathcal{V}_{1,0}^h := \mathcal{V}_1^h \cap H^1_0(\Omega\cup\Gamma),
\end{align*}
where $\mathcal{P}_{k}$ consists of all polynomial functions of degree
less than or equal to $k$. For each $\beta\in\mathcal{S}_{ad}$ the variational equations
\begin{align}
\big[u^h, \phi^h\big]_{(\boldsymbol{\alpha},\beta,\sigma)} &= \big\<f,\phi^h\big\>_\Omega + \big\<j_\delta, \phi^h\big\>_{\Gamma} + \big\<j_0, \phi^h\big\>_{\partial\Omega\setminus\Gamma}  \quad \forall \phi^h\in
\mathcal{V}_1^h \label{10/4:ct1}\\
\big[ v^h, \phi^h\big]_{(\boldsymbol{\alpha},\beta,\sigma)} &= \big\<f,\phi^h\big\>_\Omega  + \big\<j_0, \phi^h\big\>_{\partial\Omega\setminus\Gamma} \quad \forall \phi^h\in
\mathcal{V}_{1,0}^h \quad \mbox{and} \quad {v^h}_{|\Gamma} = g_\delta \label{10/4:ct1*}
\end{align}
admit unique solutions $u^h := {N}^h_{j_\delta}(\beta) \in \mathcal{V}_1^h$ and $v^h := {M}^h_{g_\delta}(\beta) \in \mathcal{V}_1^h$, respectively. Furthermore, the estimates
\begin{align}
\big\|{N}^h_{j_\delta}(\beta) \big\|_{H^1(\Omega)} 
&\preceq \|j_\delta\|_{H^{-1/2}(\Gamma)} + \|j_0\|_{H^{-1/2}(\partial\Omega\setminus\Gamma)} + \|f\|_{H^{-1}(\Omega)}  \label{18/5:ct1}\\
\big\|{M}^h_{g_\delta}(\beta) \big\|_{H^1(\Omega)} 
&\preceq \|g_\delta\|_{H^{1/2}(\Gamma)} + \|j_0\|_{H^{-1/2}(\partial\Omega\setminus\Gamma)} + \|f\|_{H^{-1}(\Omega)}  \label{18/5:ct1*}
\end{align}
hold true.

The solutions $A_N=A_N(\beta)$ and $A_M=A_M(\beta)$ of the adjoint problems \cref{17-10-16ct1*}--\cref{17-10-16ct1***+} are approximated by $A^h_N = A^h_N(\beta)\in
\mathcal{V}_{1}^h$ and $A^h_M = A^h_M(\beta)\in
\mathcal{V}_{1,0}^h$ satisfying 
\begin{align}
\big[ A^h_N(\beta), \phi^h\big]_{(\boldsymbol{\alpha},\beta,\sigma)} &= \big( {N}^h_{j_\delta}(\beta) -{M}^h_{g_\delta}(\beta), \phi^h\big)_{\Omega}  \quad \mbox{for all} \quad \phi^h\in
\mathcal{V}_{1}^h \label{10/4:ct1**}\\
\big[ A^h_M(\beta), \phi^h\big]_{(\boldsymbol{\alpha},\beta,\sigma)} &= \big( {N}^h_{j_\delta}(\beta) -{M}^h_{g_\delta}(\beta), \phi^h\big)_{\Omega}  \quad \mbox{for all} \quad \phi^h\in
\mathcal{V}_{1,0}^h. \label{10/4:ct1**+}
\end{align}
The Sobolev number of $W^m_q(\Omega)$ is defined by
$\mbox{sob}(W^m_q(\Omega)) := m-d/q.$

\begin{lemma}[Quasi-interpolation operator, see, e.g., \cite{Clement,nochetto,scott_zhang}]\label{moli.data}
There exists an operator $\Pi^h: L^1(\Omega) \rightarrow \mathcal{V}^h_{1}$ such that $\Pi^h \varphi^h = \varphi^h$
for all $\varphi^h \in \mathcal{V}^h_{1}$ and the limit
\begin{equation}\label{23/10:ct2*}
\lim_{h\to 0} \big\| \phi - \Pi^h \phi
\big\|_{W^m_q(\Omega)} =0 \quad\mbox{for all}\quad 1\le q\le\infty,~ 0\le m\le 1, \quad \phi \in W^m_q(\Omega).
\end{equation}
Furthermore, for all $T\in \mathcal{T}^h$ we have the local estimate
\begin{align}\label{17-10-16ct5}
\|\phi-\Pi^h\phi\|_{W^k_p(T)} \le C h_T^{\mbox{sob}(W^m_q(\Omega)) - \mbox{sob}(W^k_p(\Omega))} |\phi|_{W^m_q(T)},
\end{align}
where $0\le k\le m \le 2$ and $1\le p, q\le \infty$ such that $\mbox{sob}(W^m_q(\Omega)) > \mbox{sob}(W^k_p(\Omega))$. If $\phi\in W^1_1(\Omega)$ has a vanishing trace on a part $\Gamma \subset \partial\Omega$, then so does $\Pi^h\phi$. 
\end{lemma}

With the above notations at hand, the continuous regularized problem $\big(\mathcal{P}_{\delta,\rho}\big)$ can be discretized by
$$
\min_{\beta\in\mathcal{S}_{ad}} J^h_{\delta,\rho}(\beta), \quad  J^h_{\delta,\rho}(\beta) := \big\|N^h_{j_\delta}(\beta) - M^h_{g_\delta}(\beta) \big\|^2_\Omega + \rho\|\beta - \beta^*\|^2_\Omega. \eqno \big(\mathcal{P}^h_{\delta,\rho}\big)
$$ 
\begin{theorem}\label{dis-existance}
The discrete problem 
$\big(\mathcal{P}^h_{\delta,\rho}\big)$
has a minimizer $\beta^h_{\delta,\rho}$ which satisfies for a.e. in $\Omega$ the relation
\begin{align*}
\beta^h_{\delta,\rho} (x)= \mathcal{P}_{[\underline{\beta}, \overline{\beta}]} \left( \frac{1}{\rho} \big(N^h_{j_\delta}(\beta^h_{\delta,\rho})(x)A^h_N(\beta^h_{\delta,\rho})(x) - M^h_{g_\delta}(\beta^h_{\delta,\rho}) (x)A^h_M(\beta^h_{\delta,\rho})(x)\big) + \beta^*(x)\right). 
\end{align*}
\end{theorem}

\begin{proof}
The proof follows exactly as in the continuous case, we therefore omit here.
\end{proof}

%The above identity shows every minimizer of the discrete problem $\big(\mathcal{P}^h_{\delta,\rho}\big)$ belongs 
%%automatically 
%to the finite element space $\mathcal{V}^h_2$.
%%, provided that $\beta^* \in \mathcal{V}^h_2$. 
%Thus, taking this into
%account, a discretization of the admissible set $\mathcal{S}_{ad}$ can be avoided.

\section{Stability and Convergence}\label{proof}

We are in position to prove the stability of the proposed finite element method and the convergence of the regularized finite element approximations to the $\beta^*$-minimum-norm solution of the identification problem.

\begin{theorem}[Stability]\label{stability*}
Assume that the regularization parameter $\rho$ and the observation data $\big(j_\delta,g_\delta\big) \in H^{-1/2}(\Gamma) \times  H^{1/2}(\Gamma)$ are fixed. For each $n\in\mathbb{N}$ let $\beta_n := \beta^{h_n}_{\delta,\rho}$ be an arbitrary minimizer of $\big(\mathcal{P}^{h_n}_{\delta,\rho}\big)$. Then the sequence $(\beta_n)_n$ has a subsequence converging in the $L^2(\Omega)$-norm to an element $\beta_{\delta,\rho} \in \mathcal{S}_{ad}$. Furthermore, $\beta_{\delta,\rho}$ is a minimizer of $\big(\mathcal{P}_{\delta,\rho}\big)$.
\end{theorem}

\begin{proof}
In view of the proof of \Cref{weakly conv.} we deduce that a subsequence of $(\beta_n)_n$ which is not relabeled and an element $(\beta_\infty,\theta_N,\theta_M) \in \mathcal{S}_{ad} \times H^1(\Omega) \times H^1(\Omega)$ exist such that 
\begin{align*}
\beta_n &\rightharpoonup \beta_\infty ~\mbox{weakly}^* ~\mbox{in}~ L^\infty(\Omega),\\
\left( N_{j_\delta}(\beta_n),~ M_{g_\delta}(\beta_n),~ N^{h_n}_{j_\delta}(\beta_n),~ M^{h_n}_{g_\delta}(\beta_n)  \right) &\rightharpoonup \left( N_{j_\delta}(\beta_\infty),~ M_{g_\delta}(\beta_\infty),~ \theta_N,~ \theta_M\right) 
\end{align*}
weakly in $H^1(\Omega)$ as $n \to \infty$. We first show that $\theta_N = N_{j_\delta}(\beta_\infty)$, i.e. the limit $\limn\big[N_{j_\delta}^{h_n}(\beta_n) - N_{j_\delta}(\beta_\infty), \phi\big]_{\boldsymbol{\alpha},\beta_\infty,\sigma} = 0$ holds true for all $\phi\in H^1(\Omega)$. In fact, we can rewrite
\begin{align}\label{6-11-18ct1}
&\big[N_{j_\delta}^{h_n}(\beta_n), \phi\big]_{\boldsymbol{\alpha},\beta_\infty,\sigma} \notag\\
&~\quad = \big[N_{j_\delta}^{h_n}(\beta_n), \phi\big]_{\boldsymbol{\alpha},\beta_n,\sigma} + \big(N_{j_\delta}^{h_n}(\beta_n), \beta_\infty \phi\big)_\Omega - \big( N_{j_\delta}^{h_n}(\beta_n), \beta_n\phi\big)_\Omega.
\end{align}
It follows from \cref{10/4:ct1} and \cref{17-10-16ct2} that
\begin{align*}
&\big[N_{j_\delta}^{h_n}(\beta_n), \phi\big]_{\boldsymbol{\alpha},\beta_n,\sigma} = \big[N_{j_\delta}^{h_n}(\beta_n), \Pi^{h_n}\phi\big]_{\boldsymbol{\alpha},\beta_n,\sigma} + \big[N_{j_\delta}^{h_n}(\beta_n), \phi - \Pi^{h_n}\phi\big]_{\boldsymbol{\alpha},\beta_n,\sigma} \notag\\
&~\quad= \big\<f,\Pi^{h_n}\phi\big\>_\Omega + \big\<j_\delta, \Pi^{h_n}\phi\big\>_{\Gamma} + \big\<j_0, \Pi^{h_n}\phi\big\>_{\partial\Omega\setminus\Gamma} + \big[N_{j_\delta}^{h_n}(\beta_n), \phi - \Pi^{h_n}\phi\big]_{\boldsymbol{\alpha},\beta_n,\sigma}\notag\\
&~\quad= \big\<f,\phi\big\>_\Omega + \big\<j_\delta, \phi\big\>_{\Gamma} +  \big\<j_0, \phi\big\>_{\partial\Omega\setminus\Gamma} + \big[N_{j_\delta}^{h_n}(\beta_n), \phi - \Pi^{h_n}\phi\big]_{\boldsymbol{\alpha},\beta_n,\sigma}\notag\\
&~\quad\quad + \big\<f,\Pi^{h_n}\phi - \phi\big\>_\Omega + \big\<j_\delta, \Pi^{h_n}\phi - \phi\big\>_{\Gamma} + \big\<j_0, \Pi^{h_n}\phi - \phi\big\>_{\partial\Omega\setminus\Gamma} 
\end{align*}
and so 
\begin{align}\label{6-11-18ct2}
&\limn\big[N_{j_\delta}^{h_n}(\beta_n), \phi\big]_{\boldsymbol{\alpha},\beta_n,\sigma} \notag\\
&~\quad = \big\<f,\phi\big\>_\Omega + \big\<j_\delta, \phi\big\>_{\Gamma} +  \big\<j_0, \phi\big\>_{\partial\Omega\setminus\Gamma} = \big[N_{j_\delta}(\beta_\infty), \phi\big]_{\boldsymbol{\alpha},\beta_\infty,\sigma},
\end{align}
where we used \cref{18/5:ct1} and \cref{23/10:ct2*}. Furthermore, we have that
\begin{align}\label{6-11-18ct3}
&\big(N_{j_\delta}^{h_n}(\beta_n), \beta_\infty \phi\big)_\Omega - \big( N_{j_\delta}^{h_n}(\beta_n), \beta_n\phi\big)_\Omega \notag\\
&~\quad = \big(N_{j_\delta}^{h_n}(\beta_n), \beta_\infty \phi\big)_\Omega - \big(\beta_n, \theta_N\phi\big)_\Omega - \big(N_{j_\delta}^{h_n}(\beta_n) -\theta_N, \beta_n \phi \big)_\Omega \notag\\
&~\quad \to \big(\theta_N, \beta_\infty \phi\big)_\Omega - \big(\beta_\infty, \theta_N\phi\big)_\Omega \notag\\
&~\quad = 0.
\end{align}
We thus derive from \cref{6-11-18ct1}--\cref{6-11-18ct3}  
$\limn\big[N_{j_\delta}^{h_n}(\beta_n), \phi\big]_{\boldsymbol{\alpha},\beta_\infty,\sigma} = \big[N_{j_\delta}(\beta_\infty), \phi\big]_{\boldsymbol{\alpha},\beta_\infty,\sigma}$. This also yields
\begin{align}\label{6-11-18ct4}
\limn\big\|N_{j_\delta}^{h_n}(\beta_n)- N_{j_\delta}(\beta_\infty) \big\|_\Omega = 0.
\end{align}
Likewise, we can show $\theta_M = M_{g_\delta}(\beta_\infty)$ and
\begin{align}\label{6-11-18ct4*}
\limn\big\|M_{g_\delta}^{h_n}(\beta_n)- M_{g_\delta}(\beta_\infty) \big\|_\Omega = 0.
\end{align} 
Consequently, for all $\beta\in\mathcal{S}_{ad}$ we arrive at
\begin{align*}
&\big\|N_{j_\delta}(\beta_\infty)-M_{g_\delta}(\beta_\infty) \big\|^2_\Omega + \rho\|\beta_\infty - \beta^*\|^2_\Omega \\ 
&~\quad\le \limn\big\|N_{j_\delta}^{h_n}(\beta_n)-M_{g_\delta}^{h_n}(\beta_n) \big\|^2_\Omega + \liminfn  \rho\|\beta_n - \beta^*\|^2_\Omega \\
&~\quad = \liminfn \left( \big\|N_{j_\delta}^{h_n}(\beta_n)-M_{g_\delta}^{h_n}(\beta_n) \big\|^2_\Omega +\rho \|\beta_n - \beta^*\|^2_\Omega\right) \\
&~\quad \le \limsupn \left( \big\|N_{j_\delta}^{h_n}(\beta_n)-M_{g_\delta}^{h_n}(\beta_n) \big\|^2_\Omega +\rho \|\beta_n - \beta^*\|^2_\Omega\right) \\
&~\quad \le \limsupn \left( \big\|N_{j_\delta}^{h_n}(\beta)-M_{g_\delta}^{h_n}(\beta) \big\|^2_\Omega +\rho \|\beta - \beta^*\|^2_\Omega\right)\\
&~\quad =  \big\|N_{j_\delta}(\beta)-M_{g_\delta}(\beta) \big\|^2_\Omega + \rho \|\beta - \beta^*\|^2_\Omega.
\end{align*}
This means that $\beta_\infty$ is a minimizer of $\big(\mathcal{P}_{\delta,\rho}\big)$. It remains to show that $(\beta_n)_n$ converges to $\beta_\infty$ in the $L^2(\Omega)$-norm. For this purpose we take $\beta = \beta_\infty$ in the last equation to get
\begin{align}\label{6-11-18ct5}
&\limn \left( \big\|N_{j_\delta}^{h_n}(\beta_n)-M_{g_\delta}^{h_n}(\beta_n) \big\|^2_\Omega +\rho \|\beta_n - \beta^*\|^2_\Omega\right) \notag\\
&~\quad = \big\|N_{j_\delta}(\beta_\infty)-M_{g_\delta}(\beta_\infty) \big\|^2_\Omega + \rho\|\beta_\infty - \beta^*\|^2_\Omega
\end{align}
and then write
\begin{align*}
&\rho\|\beta_n - \beta_\infty\|^2_\Omega \\
&~\quad = \rho\|(\beta_n - \beta^*) - (\beta_\infty - \beta^*)\|^2_\Omega\\
&~\quad = \rho\|\beta_n - \beta^*\|^2_\Omega + \rho\|\beta_\infty - \beta^*\|^2_\Omega - 2\rho(\beta_n - \beta^*, \beta_\infty - \beta^*)_\Omega\\
&~\quad = \rho\|\beta_\infty - \beta^*\|^2_\Omega - 2\rho(\beta_n - \beta^*, \beta_\infty - \beta^*)_\Omega \\
&~\quad\quad - \big\|N_{j_\delta}^{h_n}(\beta_n)-M_{g_\delta}^{h_n}(\beta_n) \big\|^2_\Omega + \big\|N_{j_\delta}^{h_n}(\beta_n)-M_{g_\delta}^{h_n}(\beta_n) \big\|^2_\Omega +\rho \|\beta_n - \beta^*\|^2_\Omega.
\end{align*}
By the equations \cref{6-11-18ct4}--\cref{6-11-18ct5}, we obtain
\begin{align*}
&\rho\limn\|\beta_n - \beta_\infty\|^2_\Omega \\
&~\quad = \rho\|\beta_\infty - \beta^*\|^2_\Omega - 2\rho(\beta_\infty - \beta^*, \beta_\infty - \beta^*)_\Omega \\
&~\quad\quad - \big\|N_{j_\delta}(\beta_\infty)-M_{g_\delta}(\beta_\infty) \big\|^2_\Omega + \big\|N_{j_\delta}(\beta_\infty)-M_{g_\delta}(\beta_\infty) \big\|^2_\Omega +\rho \|\beta_\infty - \beta^*\|^2_\Omega\\
&~\quad =0,
\end{align*}
which finishes the proof.
\end{proof}

To go further, we remark that, due to the assumption on consistency of the system \cref{17-5-16ct1}--\cref{17-5-16ct1***}, the set
$$
\Pi_{\mathcal{S}_{ad}}(j^\dag,g^\dag) := \left\{\beta \in \mathcal{S}_{ad} ~|~ N_{j^\dag}(\beta) = M_{g^\dag}(\beta) \right\}
$$
is nonempty, convex, bounded and closed in the $L^2(\Omega)$-norm. As a result,
there exists a unique solution $\beta^\dag$ of the problem
$$\min_{\beta \in \Pi_{\mathcal{S}_{ad}}(j^\dag,g^\dag)} \|\beta - \beta^*\|_{L^2(\Omega)},$$
which is called by $\beta^*$-minimum-norm solution to the
identification problem.
Let 
\begin{align}\label{19-11-18ct3}
\varrho^h_{j_\delta,g_\delta}(\beta) := \big\|N_{j_\delta}(\beta) - N^h_{j_\delta}(\beta)\big\|_{L^2(\Omega)} + \big\|M_{g_\delta}(\beta) - M^h_{g_\delta}(\beta)\big\|_{L^2(\Omega)}.
\end{align}
Due to the standard theory of the finite element method for elliptic problems (see, e.g., \cite{Brenner_Scott}), we get 
$$\lim_{h\to 0} \varrho^h_{j_\delta,g_\delta}(\beta) = 0 \quad\mbox{and}\quad
0\le \varrho^h_{j_\delta,g_\delta}(\beta) \preceq h^2
$$
in case $N_{j_\delta}(\beta),~ M_{g_\delta}(\beta) \in H^2(\Omega)$.

\begin{theorem}[Convergence]\label{stability2}
Let
$\limn h_n = 0$. Assume that $(\delta_n)_n$ and $(\rho_n)_n$ be any positive sequences such that
\begin{align}\label{31-8-16ct1}
\rho_n \to 0, \quad \frac{\delta_n}{\sqrt{\rho_n}} \to 0, \quad \mbox{and} \quad \frac{\varrho^{h_n}_{j^\dag,g^\dag}(\beta^\dag)}{\sqrt{\rho_n}} \to 0 \quad \mbox{as} \quad n\to\infty.
\end{align} 
Furthermore, assume that
$\big(j_{\delta_n},g_{\delta_n}\big) \in H^{-1/2}(\Gamma) \times  H^{1/2}(\Gamma)$ is a sequence satisfying the inequality
\begin{align}\label{20-11-18ct2}
\big\|j_{\delta_n} - j^\dag\big\|_{H^{-1/2}(\Gamma)} + \big\|g_{\delta_n} - g^\dag\big\|_{H^{1/2}(\Gamma)}\le \delta_n
\end{align}
and $\beta_n := \beta^{h_n}_{\delta_n,\rho_n}$ denotes an arbitrary minimizer of $\big(\mathcal{P}^{h_n}_{\delta_n,\rho_n}\big)$ for each $n\in \mathbb{N}$. Then: 

(i) The whole sequence $(\beta_n)_n$ converges in the $L^2(\Omega)$-norm to $\beta^\dag$.

(ii) The corresponding state sequences $\big( N^{h_n}_{j_{\delta_n}}(\beta_n)\big)_n$ and $\big( M^{h_n}_{g_{\delta_n}}(\beta_n)\big)_n$ converge in the $H^1(\Omega)$-norm to the exact state $\Phi^\dag = \Phi(j^\dag,g^\dag,\beta^\dag)$ of the problem \cref{17-5-16ct1}--\cref{17-5-16ct1***}.
\end{theorem}

Note that in case the exact solution $\Phi^\dag \in H^2(\Omega)$ (cf.\ \Cref{27-12-18ct2}) then the convergences (i) and (ii) are obtained if the regularization parameter is chosen such that $\rho_n\to 0$, $\delta_n/\sqrt{\rho_n} \to 0$ and $h^2_n/\sqrt{\rho_n} \to 0$ as $n\to \infty$.
\begin{proof}[Proof of \Cref{stability2}]
We have from the optimality of $\beta_n$ for each $n$ that
\begin{align}\label{6-5-16ct8}
&\big\|N^{h_n}_{j_{\delta_n}}(\beta_n)-M^{h_n}_{g_{\delta_n}}(\beta_n) \big\|^2_{\Omega} + \rho_n \|\beta_n-\beta^*\|^2_{\Omega} \notag\\
&~\quad \le \big\|N^{h_n}_{j_{\delta_n}}(\beta^\dag)-M^{h_n}_{g_{\delta_n}}(\beta^\dag) \big\|^2_{\Omega} + \rho_n \|\beta^\dag - \beta^*\|^2_{\Omega}.
\end{align}
Note that $N_{j^\dag}(\beta^\dag) = M_{g^\dag}(\beta^\dag)$, we thus have
\begin{align*}
&\big\|N^{h_n}_{j_{\delta_n}}(\beta^\dag)-M^{h_n}_{g_{\delta_n}}(\beta^\dag) \big\|_{\Omega}\\ 
&~\quad = \big\|N^{h_n}_{j_{\delta_n}}(\beta^\dag) - N_{j^\dag}(\beta^\dag) + M_{g^\dag}(\beta^\dag) -M^{h_n}_{g_{\delta_n}}(\beta^\dag) \big\|_{\Omega}\notag\\
&~\quad \le \big\|N^{h_n}_{j_{\delta_n}}(\beta^\dag) - N^{h_n}_{j^\dag}(\beta^\dag) \big\|_{\Omega} + \big\|N^{h_n}_{j^\dag}(\beta^\dag) - N_{j^\dag}(\beta^\dag)\big\|_{\Omega} \notag\\
&~\quad\quad + \big\|M^{h_n}_{g_{\delta_n}}(\beta^\dag) - M^{h_n}_{g^\dag}(\beta^\dag) \big\|_{\Omega} + \big\|M^{h_n}_{g^\dag}(\beta^\dag) - M_{g^\dag}(\beta^\dag)\big\|_{\Omega}\notag\\
&~\quad = \varrho^{h_n}_{j^\dag,g^\dag}(\beta^\dag) + \big\|N^{h_n}_{j_{\delta_n}}(\beta^\dag) - N^{h_n}_{j^\dag}(\beta^\dag) \big\|_{\Omega} + \big\|M^{h_n}_{g_{\delta_n}}(\beta^\dag) - M^{h_n}_{g^\dag}(\beta^\dag) \big\|_{\Omega}.
\end{align*}
By the identities \cref{10/4:ct1}--\cref{10/4:ct1*} and the inequality \cref{20-11-18ct2}, we get
\begin{align}\label{20-11-18ct4}
&\big\|N^{h_n}_{j_{\delta_n}}(\beta^\dag) - N^{h_n}_{j^\dag}(\beta^\dag) \big\|_{\Omega} + \big\|M^{h_n}_{g_{\delta_n}}(\beta^\dag) - M^{h_n}_{g^\dag}(\beta^\dag) \big\|_{\Omega} \notag\\
&~\quad \preceq \big\|j_{\delta_n} - j^\dag\big\|_{H^{-1/2}(\Gamma)} + \big\|g_{\delta_n} - g^\dag\big\|_{H^{1/2}(\Gamma)} \notag\\
&~\quad\le \delta_n
\end{align}
which yields
\begin{align}\label{20-11-18ct1}
\big\|N^{h_n}_{j_{\delta_n}}(\beta^\dag)-M^{h_n}_{g_{\delta_n}}(\beta^\dag) \big\|_{\Omega} \preceq \varrho^{h_n}_{j^\dag,g^\dag}(\beta^\dag) + \delta_n.
\end{align}
It follows from \cref{6-5-16ct8}--\cref{20-11-18ct1} that
\begin{align}\label{6-5-16ct10}
\limn \big\|N^{h_n}_{j_{\delta_n}}(\beta_n)-M^{h_n}_{g_{\delta_n}}(\beta_n) \big\|_{\Omega}  =0
\end{align}
and
\begin{align}\label{6-5-16ct11}
\limsupn \|\beta_n-\beta^*\|_{\Omega} \le  \|\beta^\dag-\beta^*\|_{\Omega}.
\end{align}
Next, we get that
\begin{align*}
&\big\|N^{h_n}_{j^\dag}(\beta_n)-M^{h_n}_{g^\dag}(\beta_n) \big\|_{\Omega}  \\
& =  \big\|N^{h_n}_{j^\dag}(\beta_n) - N^{h_n}_{j_{\delta_n}}(\beta_n) + N^{h_n}_{j_{\delta_n}}(\beta_n) - M^{h_n}_{g_{\delta_n}}(\beta_n) + M^{h_n}_{g_{\delta_n}}(\beta_n) - M^{h_n}_{g^\dag}(\beta_n) \big\|_{\Omega} \notag\\
& \le \big\|N^{h_n}_{j^\dag}(\beta_n) - N^{h_n}_{j_{\delta_n}}(\beta_n)\big\|_\Omega + \big\| M^{h_n}_{g_{\delta_n}}(\beta_n) - M^{h_n}_{g^\dag}(\beta_n) \big\|_{\Omega} + \big\|N^{h_n}_{j_{\delta_n}}(\beta_n)-M^{h_n}_{g_{\delta_n}}(\beta_n) \big\|_{\Omega} \notag\\
& \preceq \delta_n + \big\|N^{h_n}_{j_{\delta_n}}(\beta_n)-M^{h_n}_{g_{\delta_n}}(\beta_n) \big\|_{\Omega} \\
&\to 0
\end{align*}
as $n\to\infty$, by \cref{6-5-16ct10}. Furthermore, since $\beta_n\in\mathcal{S}_{ad}$ for all $n\in \mathbb{N}$, in view of the proof of \Cref{stability*}, a subsequence of $\left( \beta_n\right)_n $ not relabeled and an element $\widehat{\beta} \in \mathcal{S}_{ad}$ exist such that 
\begin{align}
& \beta_n \rightharpoonup \widehat{\beta} \mbox{~ weakly}^* \mbox{~ in ~} L^\infty(\Omega), \nonumber\\
& \|\widehat{\beta}-\beta^*\|_{\Omega} \le \liminfn \|\beta_n-\beta^*\|_{\Omega}, \label{6-5-16ct12}\\
& \left(N^{h_n}_{j^\dag}(\beta_n),~ M^{h_n}_{g^\dag}(\beta_n)\right) \rightharpoonup \left( N_{j^\dag}(\widehat{\beta}),~ M_{g^\dag}(\widehat{\beta})\right) ~\mbox{weakly in}~ H^1(\Omega)\notag
\end{align}
which also implies that
\begin{align*}
\big\|N_{j^\dag}(\widehat{\beta}) - M_{g^\dag}(\widehat{\beta})\big\|_\Omega = \limn \big\|N^{h_n}_{j^\dag}(\beta_n)-M^{h_n}_{g^\dag}(\beta_n) \big\|_{\Omega} = 0
\end{align*}
and so that $\widehat{\beta} \in \Pi_{\mathcal{S}_{ad}}(j^\dag,g^\dag)$. Then,
combining \cref{6-5-16ct12} with \cref{6-5-16ct11} and using the uniqueness of the $\beta^*$-minimum-norm solution, we obtain
\begin{align}\label{9-11-18ct1}
\widehat{\beta} = \beta^\dag \mbox{~and~} \limn \|\beta_n - \widehat{\beta}\|_{\Omega} = 0,
\end{align}
the assertion (i) is thus proved. For (ii) we have for all $\phi^{h_n}\in\mathcal{V}^{h_n}_1$ that
\begin{align*}
\big[N^{h_n}_{j_{\delta_n}}(\beta^\dag) - N^{h_n}_{j_{\delta_n}}(\beta_n),\phi^{h_n}\big]_{\boldsymbol{\alpha},\beta^\dag,\sigma} 
=  \big(\beta_n - \beta^\dag, N^{h_n}_{j_{\delta_n}}(\beta_n) \phi^{h_n}\big) 
\end{align*}
which together with the inequality \cref{18-10-16ct1} and the continuous embedding $H^1(\Omega) \hookrightarrow L^4(\Omega)$  imply
\begin{align*}
&\big\|N^{h_n}_{j_{\delta_n}}(\beta^\dag) - N^{h_n}_{j_{\delta_n}}(\beta_n) \big\|^2_{1,\Omega}\\
&~\quad \le \big\|\beta_n - \beta^\dag\big\|_\Omega \big\|N^{h_n}_{j_{\delta_n}}(\beta_n)\big\|_{L^4(\Omega)} \big\|N^{h_n}_{j_{\delta_n}}(\beta^\dag) - N^{h_n}_{j_{\delta_n}}(\beta_n)\big\|_{L^4(\Omega)}\\
&~\quad \preceq \big\|\beta_n - \beta^\dag\big\|_\Omega \big\|N^{h_n}_{j_{\delta_n}}(\beta_n)\big\|_{1,\Omega} \big\|N^{h_n}_{j_{\delta_n}}(\beta^\dag) - N^{h_n}_{j_{\delta_n}}(\beta_n) \big\|_{1,\Omega}\\
&~\quad \preceq \big\|\beta_n - \beta^\dag\big\|_\Omega \big\|N^{h_n}_{j_{\delta_n}}(\beta^\dag) - N^{h_n}_{j_{\delta_n}}(\beta_n) \big\|_{1,\Omega}.
\end{align*} 
Therefore, we arrive at
\begin{align}\label{20-11-18ct5}
\limn \big\|N^{h_n}_{j_{\delta_n}}(\beta^\dag) - N^{h_n}_{j_{\delta_n}}(\beta_n) \big\|_{1,\Omega}
\preceq \limn \big\|\beta_n - \beta^\dag\|_\Omega =0
\end{align}  
with the aid of the limit \cref{9-11-18ct1}. Consequently, we obtain from  \cref{20-11-18ct4} and \cref{20-11-18ct5} that
\begin{align*}
&\big\|N^{h_n}_{j_{\delta_n}}(\beta_n)  - N_{j^\dag}(\beta^\dag) \big\|_{1,\Omega} \\
&~\quad \le \big\|N^{h_n}_{j_{\delta_n}}(\beta_n)  - N^{h_n}_{j_{\delta_n}}(\beta^\dag) \big\|_{1,\Omega} + \big\|N^{h_n}_{j_{\delta_n}}(\beta^\dag) - N^{h_n}_{j^\dag}(\beta^\dag)\big\|_{1,\Omega} \\
&~\quad \quad + \big\|N^{h_n}_{j^\dag}(\beta^\dag) - N_{j^\dag}(\beta^\dag)\big\|_{1,\Omega}\\
&~\quad \to 0
\end{align*}
as $n\to\infty$. By the similar arguments, we also get $\limn \big\|M^{h_n}_{g_{\delta_n}}(\beta_n)  - M_{g^\dag}(\beta^\dag) \big\|_{1,\Omega} =0$, which finishes the proof.
\end{proof}

\section{Gradient projection algorithm and numerical implementation}\label{iterative}

\subsection{Algorithm}\label{PAA}

In this section we present the
gradient projection algorithm with Armijo steplength rule (cf.\ \cite{kelley,rusz}) for numerical solution of the minimization  problem $\big(\mathcal{P}^{h}_{\delta,\rho} \big)$.

In view of the proof of \Cref{existance}, 
we first note that for each $\beta\in \mathcal{S}_{ad}$ the $L^2$-gradient of the cost functional $J^h_{\delta,\rho}$ of the problem $\big( \mathcal{P}^h_{\delta,\rho} \big)$ at $\beta$ is given by 
$$\nabla J^h_{\delta,\rho} (\beta) = M^h_{g_\delta}(\beta)A^h_M(\beta) - N^h_{j_\delta}(\beta)A^h_N(\beta) + \rho(\beta-\beta^*).$$
The algorithm is then read as: given a step size control $\mu \in (0,1)$, an initial approximation $\beta_0$, number of iteration $N$ and setting $k=0$.
\begin{enumerate}
\item Compute $N^h_{j_\delta}(\beta_k)$ and $M^h_{g_\delta}(\beta_k)$  from the equations
\begin{align*}
\big[N^h_{j_\delta}(\beta_k), \phi^h\big]_{(\boldsymbol{\alpha},\beta_k,\sigma)} &= \big\<f,\phi^h\big\>_\Omega + \big\<j_\delta, \phi^h\big\>_{\Gamma} + \big\<j_0, \phi^h\big\>_{\partial\Omega\setminus\Gamma}  \quad \forall \phi^h\in
\mathcal{V}_1^h \\
\big[ M^h_{g_\delta}(\beta_k), \phi^h\big]_{(\boldsymbol{\alpha},\beta_k,\sigma)} &= \big\<f,\phi^h\big\>_\Omega  + \big\<j_0, \phi^h\big\>_{\partial\Omega\setminus\Gamma} \quad \forall \phi^h\in
\mathcal{V}_{1,0}^h, \quad {v^h}_{|\Gamma} = g_\delta 
\end{align*}
and then the solutions $A^h_N(\beta_k)$ and $A^h_M(\beta_k)$ of the adjoint problems
\begin{align*}
\big[ A^h_N(\beta_k), \phi^h\big]_{(\boldsymbol{\alpha},\beta_k,\sigma)} 
&= \big( {N}^h_{j_\delta}(\beta_k) -{M}^h_{g_\delta}(\beta_k), \phi^h\big)_{\Omega}  \quad \mbox{for all} \quad \phi^h\in
\mathcal{V}_{1}^h\\
\big[ A^h_M(\beta_k), \phi^h\big]_{(\boldsymbol{\alpha},\beta_k,\sigma)} 
&= \big( {N}^h_{j_\delta}(\beta_k) -{M}^h_{g_\delta}(\beta_k), \phi^h\big)_{\Omega}  \quad \mbox{for all} \quad \phi^h\in
\mathcal{V}_{1,0}^h.
\end{align*} 
\item Compute the corresponding value of the  cost functional
$$J^h_{\delta,\rho}(\beta_k) := \big\|N^h_{j_\delta}(\beta_k) - M^h_{j_\delta}(\beta_k) \big\|^2_\Omega + \rho\|\beta_k - \beta^*\|^2_\Omega$$
as well as the gradient 
\begin{align*}
\nabla J^h_{\delta,\rho} (\beta_k) = M^h_{g_\delta}(\beta_k)A^h_M(\beta_k)  - N^h_{j_\delta}(\beta_k)A^h_N(\beta_k) + \rho(\beta_k-\beta^*).
\end{align*}
\item Compute 
$$\widehat{\beta}_k := \max\left( \underline{\beta}, \min\big( \overline{\beta}, \beta_k - \mu \nabla J^h_{\delta,\rho} (\beta_k)\big)\right)$$
and then the corresponding states $N^h_{j_\delta}(\widehat{\beta}_k)$ and $M^h_{g_\delta}(\widehat{\beta}_k)$, the value of the cost functional $J^h_{\delta,\rho}(\widehat{\beta}_k)$ as well.

\item Compute the quantity
$$Q := J^h_{\delta,\rho}(\beta_k) - J^h_{\delta,\rho}(\widehat{\beta}_k)+ \tau \mu \|\widehat{\beta}_k - \beta_k\|^2_\Omega$$
for a small positive constant $\tau=10^{-4}$.
\begin{enumerate}
\item If $Q\ge 0$

\qquad  go to the next step (b) below

else

\qquad  set $\mu := \frac{\mu}{2}$ and then go back to the step 3.
\item Update $\beta_k = \widehat{\beta}_k$, set $k=k+1$.
\end{enumerate}
\item  Compute
\begin{align}\label{19-12-18ct1}
\mbox{Tolerance} := \big\| \nabla J^h_{\delta,\rho}(\beta_k) \big\|_\Omega  -\tau_1 -\tau_2\big\| \nabla J^h_{\delta,\rho}(\beta_0) \big\|_\Omega
\end{align}
for $\tau_1 := 10^{-3}h$ and $\tau_2 := 10^{-2}h$. If $\mbox{Tolerance} \le 0$ or $k>N$, then stop; otherwise go back to the step 1.
\end{enumerate}

\subsection{Numerical implementation}

For illustrating the theoretical result we consider the system \cref{17-5-16ct1}--\cref{17-5-16ct1***}
with $\Omega = \{ x = (x_1,x_2) \in \mathbb{R}^2 ~|~ -1 < x_1, x_2 < 1\}$. For discretization
we divide the interval $(-1,1)$ into $\ell$ equal segments, and so the domain $\Omega = (-1,1)^2$ is divided into $2\ell^2$ triangles, where the diameter of each triangle is $h_{\ell} = \frac{\sqrt{8}}{\ell}$.

The source function $f$ is assumed to be discontinuous and defined as
$$f := \frac{3}{2} \chi_{D} - \frac{1}{2}\chi_{\Omega \setminus D},$$ 
where $\chi_D$ is the characteristic function of the Lebesgue measurable set 
$$D := \left\{ (x_1, x_2) \in \Omega ~\big|~ |x_1|\le 1/2 \quad\mbox{and}\quad |x_2| \le 1/2\right\}.$$
We assume that entries of the symmetric diffusion matrix $\boldsymbol{\alpha}$ are discontinuous which are defined as
$$\alpha_{11} := 2\chi_{\Omega_{11}} + \chi_{\Omega\setminus\Omega_{11}},\quad 
\alpha_{12} = \alpha_{21} := \chi_{\Omega_{12}} \quad\mbox{and}\quad
\alpha_{22} := 3\chi_{\Omega_{22}} + 2\chi_{\Omega\setminus\Omega_{22}},$$
with
\begin{align*}
\Omega_{11} &:= \left\{ (x_1, x_2) \in \Omega ~\big|~ |x_1| \le 3/4 \mbox{~and~} |x_2| \le 3/4 \right\},\\
\Omega_{12} &:= \left\{ (x_1, x_2) \in \Omega ~\big|~ |x_1| + |x_2| \le 3/4 \right\} \quad\mbox{and}\\
\Omega_{22} &:= \left\{ (x_1, x_2) \in \Omega ~\big|~ x_1^2 + x_2^2 \le  9/16 \right\}.
\end{align*}
Furthermore, the special function $\sigma$ is chosen to be zero while the constants appearing in the admissible set $\mathcal{S}_{ad}$ defined by (\ref{2-11-18ct1}) are chosen as $\underline{\beta} = 0.05$ and $\overline{\beta}=10$. 

The sought reaction coefficient $\beta^\dag$ is assumed to be discontinuous and given by 
$$\beta^\dag := 3\chi_{\Omega_0} + \chi_{\Omega \setminus \Omega_0}$$
with 
$$\Omega_0:= \left\{ (x_1, x_2) \in \Omega ~\big|~ 4x_1^2 + 9x_2^2 \le  1 \right\}.$$
The Neumann boundary condition on the bottom and left surface is given by
\begin{align}\label{22-12-18ct1}
j^\dag := A\cdot \chi_{(0,1]\times\{-1\}} + B\cdot \chi_{[-1,0]\times\{-1\}} + C\cdot \chi_{\{-1\}\times(-1,0]} + D\cdot \chi_{\{-1\}\times(0,1)}
\end{align}
and on the right and top surface
$$j_0 := 4\chi_{\{1\}\times(-1,0]} - 3\chi_{\{1\}\times(0,1)} + 2\chi_{(0,1]\times\{1\}} - \chi_{[-1,0]\times\{1\}}$$
with the constants $A, B, C$ and $D$ discussed in details later.
The exact state $\Phi^\dag$ is then computed from the finite element equation $K\Phi^\dag=F$, where $K$ and $F$ are the stiffness matrix and the load vector associated with the problem \cref{17-5-16ct1}--\cref{17-5-16ct1**}, respectively. The Dirichlet boundary condition $g^\dag$ in \cref{17-5-16ct1***} is then defined as $g^\dag = \gamma_\Gamma\Phi^\dag$, the Dirichlet trace of $\Phi^\dag$ on the boundary $\Gamma$.

We use the algorithm which is described in \Cref{PAA} for computing the numerical solution of the problem $\big(\mathcal{P}_{\delta,\rho}^{h} \big)$. The step size control is chosen with $\mu=0.75$. The initial approximation and an a priori estimate are the constant functions defined by $\beta_0 = \beta ^* = 1.5$. Our computational process will be started with the coarsest level $\ell=4$. In each iteration $k$ we compute Tolerance
defined by \cref{19-12-18ct1}. Then the iteration is stopped if
$\mbox{Tolerance} \le 0$ or the number of iterations reaches the maximum iteration counted of 600.
After obtaining the numerical solution of the first iteration process with respect to the coarsest level $\ell=4$, we use its interpolation on the next finer mesh $\ell=8$ as the initial approximation and an a priori estimate as well for the algorithm on this finer mesh, and so on for $\ell=16,32,64$.

We mention that in our numerical implementation the sought reaction coefficient is chosen to be discontinuous. To reconstruct such a discontinuous function one usually employs the total variation regularization. We will discuss the details in the last section \S \ref{Conclusions}.

\begin{exam}\label{exam1}
In this example we assume $(A,B,C,D) = (1,-2,3,-4)$ while observations are taken on the bottom surface $\Gamma_{\mbox{observation}} := \Gamma_{\mbox{bottom}} := [-1,1]\times\{-1\}$ only. We assume that noisy observations are available in the form
\begin{align}\label{3-7-17ct1}
\left( j_{\delta_{\ell}}, g_{\delta_{\ell}} \right) = \left( j^\dag+ \theta_\ell\cdot R_{j^\dag}, g^\dag+ \theta_\ell\cdot R_{g^\dag}\right) 
\end{align}
for some $\theta_\ell>0$ depending on  $\ell$, where $R_{j^\dag}$ and $R_{g^\dag}$ are $\partial M^{h_\ell}\times 1$-matrices of random numbers on the interval $(-1,1)$ which are generated by the MATLAB function ``rand'' and $\partial M^{h_\ell}$ is the set of boundary nodes of the triangulation $\mathcal{T}^{h_\ell}$ which belong to $\Gamma_{\mbox{observation}}$. The measurement error is then computed as $\delta_\ell = \big\|j_{\delta_\ell} -j^\dag\big\|_{L^2(\Gamma)} + \big\|g_{\delta_\ell} -g^\dag\big\|_{L^2(\Gamma)}.$ 
To satisfy the condition \cref{31-8-16ct1} in \Cref{stability2} we below take $\theta_\ell = h_\ell \sqrt{10\cdot\rho_\ell}$ and the regularization parameter $\rho = \rho_\ell = 0.001 \sqrt{h_\ell}$.

Let $\beta_\ell$ denote the reaction coefficient obtained at the final iteration of the algorithm corresponding to the refinement level $\ell$. We then use the following abbreviations for the errors
\begin{align*}
&L^2_\beta = \big\|\beta_\ell -  \beta^\dag\big\|_\Omega, \en L^2_N = \big\|N^{h_\ell}_{j_{\delta_\ell}} ({\beta_\ell})  - N^{h_\ell}_{j^\dag}(\beta^\dag)\big\|_\Omega,\\
&L^2_M = \big\|M^{h_\ell}_{j_{\delta_\ell}} ({\beta_\ell})  - M^{h_\ell}_{j^\dag}(\beta^\dag)\big\|_\Omega, \en L^2_D = \Big\|D^{h_\ell}_{g^\dag_{\delta_\ell}} ({\beta_\ell})  - D^{h_\ell}_{\widehat{g}^\dag}(\beta^\dag)\Big\|_\Omega,
\end{align*}
where 
$$g^\dag_{\delta_\ell} =
\begin{cases}
g_{\delta_\ell} & \mbox{on} \quad \Gamma= \Gamma_{\mbox{observation}},\\
\gamma_{\partial\Omega\setminus\Gamma}\Phi^\dag & \mbox{on} \quad \partial\Omega\setminus \Gamma
\end{cases} \quad \mbox{and} \quad \widehat{g}^\dag := \gamma_{\partial\Omega}\Phi^\dag,$$
$D^{h_\ell}_{g^\dag_{\delta_\ell}} ({\beta_\ell})$ and $D^{h_\ell}_{\widehat{g}^\dag}(\beta^\dag)$ are numerical solutions of the problem \eqref{17-5-16ct1} with $\beta = \beta_\ell$ and $\beta = \beta^\dag$, respectively supplemented with the Dirichlet boundary condition $\Phi_{|\partial\Omega} = g^\dag_{\delta_\ell}$ and $\Phi_{|\partial\Omega} = \widehat{g}^\dag$.

The numerical result is summarized in \Cref{b1} and \Cref{b11}, where we present the refinement level $\ell$, the mesh size $h_\ell$ of the triangulation, the regularization parameter $\rho_\ell$, the measurement noise $\delta_\ell$, and the errors $L^2_\beta$, $L^2_N$, $L^2_M$, $L^2_D$.

\begin{table}[H]
\begin{center}
\begin{tabular}{|c|l|l|l|l|l|l|l|}
\hline \multicolumn{8}{|c|}{ { Error history for $\Gamma_{\mbox{observation}} := \Gamma_{\mbox{bottom}}$} 
\vspace{0.05cm}
}\\

\hline
$\ell$ &\scriptsize $h_\ell$ &\scriptsize $\rho_\ell$ &\scriptsize $\delta_\ell$ &\scriptsize {$L^2_\beta$} &\scriptsize {$L^2_N$} &\scriptsize {$L^2_M$} &\scriptsize {$L^2_D$}\\
\hline
4   &0.7071 & 8.4090e-4& 0.1116&  1.2185 & 0.3026& 0.2674 &0.1289\\
\hline
8  &0.3536 & 5.9460e-4& 4.0042e-2&  0.5989& 0.1377& 0.1205 &9.4845e-2\\
\hline
16  &0.1767 & 4.2045e-4& 1.9375e-2&  0.2378& 0.1014& 7.7931e-2 &3.6737e-2\\
\hline
32  &8.8388e-2 & 2.9730e-4& 7.9021e-3&  0.1472& 4.6465e-2& 3.5284e-2 &1.6411e-2\\
\hline
64  &4.4194e-2 & 2.1022e-4& 3.2289e-3& 7.4264e-2& 2.0169e-2& 1.6636e-2 &6.9882e-3\\
\hline
\end{tabular}
\end{center}
\caption{Refinement level $\ell$, mesh size $h_\ell$, regularization parameter $\rho_\ell$,  measurement noise $\delta_\ell$, and errors $L^2_\beta$, $L^2_N$, $L^2_M$, $L^2_D$.}
\label{b1}
\end{table}

The experimental order of convergence (EOC) is presented in Table \ref{b11}, where
$$\mbox{EOC}_\Theta := \dfrac{\ln \Theta(h_1) - \ln \Theta(h_2)}{\ln h_1 - \ln h_2}$$
and $\Theta(h)$ is an error function with respect to the mesh size $h$.

\begin{table}[H]
\begin{center}
\begin{tabular}{|c|c|c|c|c|}
 \hline \multicolumn{5}{|c|}{ { Experimental order of convergence \big($\Gamma_{\mbox{observation}} := \Gamma_{\mbox{bottom}}$\big)} }\\
 \hline
$\ell$ &\scriptsize { EOC$_{L^2_\beta}$} &\scriptsize { EOC$_{L^2_N}$} &\scriptsize { EOC$_{L^2_M}$} &\scriptsize { EOC$_{L^2_D}$}\\
\hline
4 & --& -- & -- &--\\
\hline
8 & 1.0247 & 1.1359 & 1.1500& 0.4426\\
\hline
16 & 1.3326 & 0.4415 & 0.6288 &1.3683 \\
\hline
32 & 0.6920 & 1.1258 & 1.1432 &1.1626\\
\hline
64 & 0.9870 & 1.2040 & 1.0847 &1.2317 \\
\hline
Mean of EOC 
& 1.0091 & 0.9768 & 1.0017 & 1.0513 \\
\hline
\end{tabular}
\caption{Experimental order of convergence between finest and coarsest
level for $L^2_\beta$, $L^2_N$, $L^2_M$ and $L^2_D$.}
\label{b11}
\end{center}
\end{table}

The error history given in \Cref{b1} shows that the algorithm performs well for our identification problem. We also observe a decrease of all errors as the noise level and the mesh size gets smaller as expected from our convergence result in \Cref{stability2}.

All figures presented hereafter correspond to the finest level $\ell = 64$. \Cref{h1} from left to right shows the numerical solution $\beta_\ell$  computed by the algorithm at the final 523$^{\mbox{\tiny th}}$-iteration and the differences $\beta_\ell - I^{h_\ell}_1 \beta^\dag $, $N^{h_\ell}_{j^\dag}(\beta^\dag) - N^{h_\ell}_{j_{\delta_\ell}} ({\beta_\ell})$, $D^{h_\ell}_{{\widehat{g}^\dag}}(\beta^\dag) - D^{h_\ell}_{g^\dag_{\delta_\ell}} ({\beta_\ell})$, where $I^{h_\ell}_1$ is the usual Lagrange node value interpolation operator.

\begin{figure}[H]
\begin{center}
\includegraphics[scale=0.16]{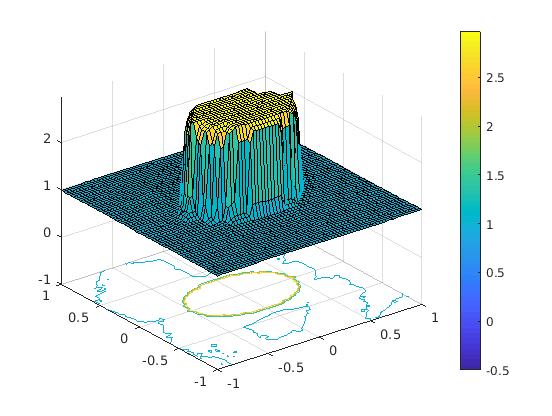}
\includegraphics[scale=0.16]{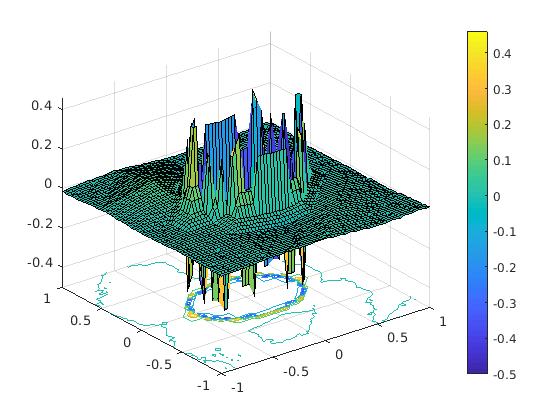} 
\includegraphics[scale=0.16]{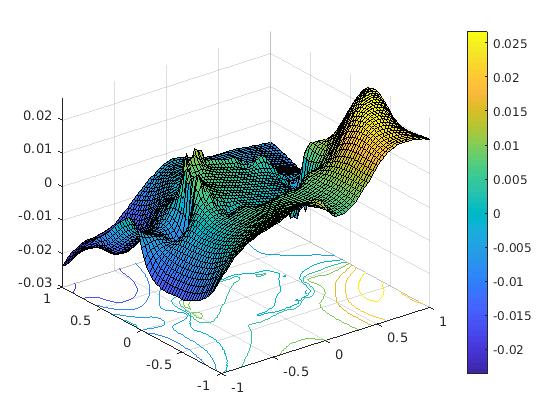}
\includegraphics[scale=0.16]{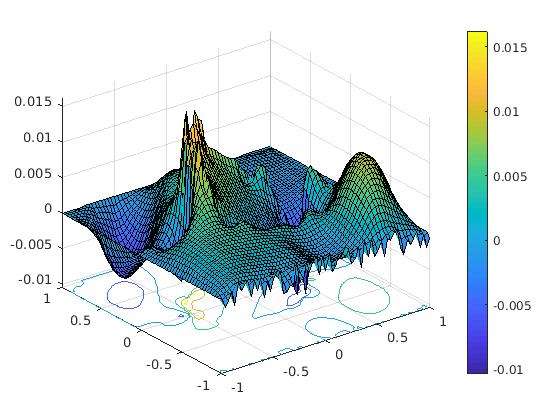}
\end{center}
\caption{Computed numerical solution $\beta_\ell$ of the algorithm at the final 523$^{\mbox{\tiny th}}$-iteration, and the differences $\beta_\ell - I^{h_\ell}_1 \beta^\dag $, $N^{h_\ell}_{j^\dag}(\beta^\dag) - N^{h_\ell}_{j_{\delta_\ell}} ({\beta_\ell})$, $D^{h_\ell}_{{\widehat{g}^\dag}}(\beta^\dag) - D^{h_\ell}_{g^\dag_{\delta_\ell}} ({\beta_\ell})$ for $\ell=64$, $\rho = \rho_\ell = 0.001 \sqrt{h_\ell}$ and $\delta_\ell=3.2289$e-3.}
\label{h1}
\end{figure}
\end{exam}

%=============================
%=============================

\begin{exam}\label{exam2}
In this example we consider regularization parameter in the forms
$$\rho_\ell = h_\ell^2,~ 0.1h_\ell^2,~ 0.01h_\ell^2,~ \mbox{and}~ 0.001h_\ell^2,$$
while $\theta_\ell = 10^{-2} h_\ell \sqrt{\rho_\ell}$. Using the computational process which was described as in \Cref{exam1} starting with $\ell=4$, in \Cref{b2} we perform the numerical result for the finest grid $\ell=64$ and with different values of $\rho_\ell$ above.

\begin{table}[H]
\begin{center}
\begin{tabular}{|c|l|l|l|l|l|}
\hline
\scriptsize { Regularization parameter} &\scriptsize {$\delta_\ell$} &\scriptsize {$L^2_\beta$} &\scriptsize {$L^2_N$} &\scriptsize {$L^2_M$} &\scriptsize {$L^2_D$}\\
\hline
$\rho_\ell= h_\ell^2$e-0 = $1.9531$e-3 & 3.4240e-5 & 5.0702 &  1.5779& 1.1618 & 0.1805\\
\hline
$\rho_\ell= h_\ell^2$e-1 = $1.9531$e-4 & 9.3133e-6 & 2.4302 & 1.1983& 0.8792& 0.1439\\
\hline
$\rho_\ell= h_\ell^2$e-2 = $1.9531$e-5 & 3.3618e-6 & 0.1028 & 3.6218e-2& 2.9314e-2 & 1.5681e-2\\
\hline
$\rho_\ell= h_\ell^2$e-3 = $1.9531$e-6 & 1.0207e-6 & 8.1651e-2 & 3.4587e-2& 2.7745e-2 & 1.3142e-2\\
\hline
\end{tabular}
\end{center}
\caption{Regularization parameter $\rho_\ell$,  measurement noise $\delta_\ell$ and errors $L^2_\beta$, $L^2_N$, $L^2_M$, $L^2_D$.}
\label{b2}
\end{table}

We note that the numerical result is not so good in first two cases, where $\rho_\ell = h_\ell^2$ and $\rho_\ell = 0.1h_\ell^2$. This indicates that previous iteration processes corresponding to the coarse grid levels $\ell=4, 8, ...$ strongly affect  the final obtained numerical result with respect to the finest grid level $\ell=64$. 

A crucial problem of Tikhonov regularization as well as other regularization methods is the choice of the regularization parameter. If the regularization parameter is too large, one obtains only
a poor approximation of the exact solution even for exact data (i.e. the error level of observations is zero), while if it is too small, the reconstruction becomes slow or unstable. The discussion on this subject is still ongoing and for a survey the reader may consult \cite{EHN96,kir96,Mor84}. We mention that by the relation \eqref{31-8-16ct1}, if $\delta_\ell\rho_\ell^{-1/2} \to 0$ the regularized approximation $\beta_\ell$ is convergent in the $L^2(\Omega)$-norm to the sought coefficient $\beta^\dag$ if $\rho_\ell$ is chosen to be $h^2_\ell$. However, the computation given in \Cref{b2} shows in case $\rho_\ell = h^2_\ell$ and $\rho_\ell = 0.1h^2_\ell$ we get only poor approximations of the sought coefficient, though the data noise levels $\delta_\ell$ are quite small.

With $\rho_\ell = h_\ell^2$, \Cref{h2} from left to right shows the computed numerical solution $\beta_\ell$  at the final iteration, the differences $N^{h_\ell}_{j^\dag}(\beta^\dag) - N^{h_\ell}_{j_{\delta_\ell}} ({\beta_\ell})$ and $D^{h_\ell}_{\widehat{g}^\dag}(\beta^\dag) - D^{h_\ell}_{g^\dag_{\delta_\ell}} ({\beta_\ell})$, while the differences $\beta_\ell - I^{h_\ell}_1 \beta^\dag$ and $\beta_\ell(x_1,0) - I^{h_\ell}_1 \beta^\dag(x_1,0)$ as the second variable $x_2=0$ are shown respectively in left two figures of \Cref{h3}. The right one of \Cref{h3} performs the graphs $\beta_\ell(x_1,0)$ (blue one) and $I^{h_\ell}_1 \beta^\dag(x_1,0)$. 

\begin{figure}[H]
\begin{center}
\includegraphics[scale=0.21]{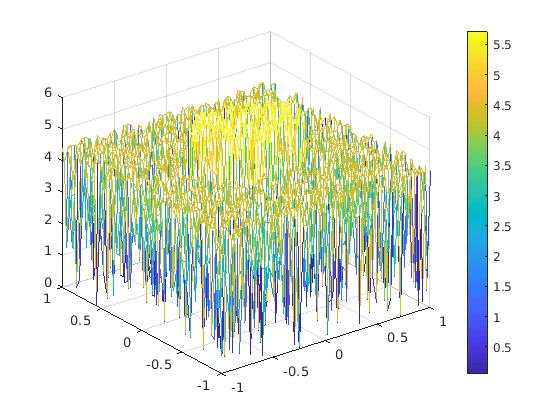}
\includegraphics[scale=0.21]{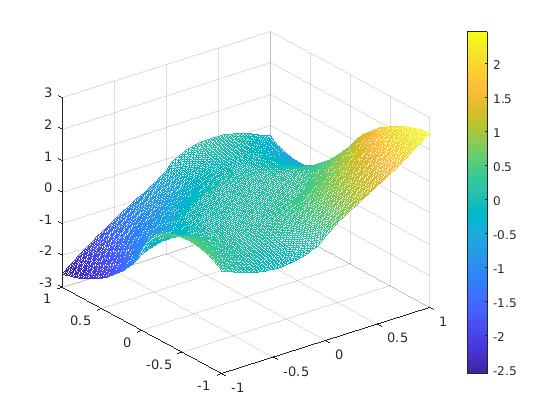}
\includegraphics[scale=0.21]{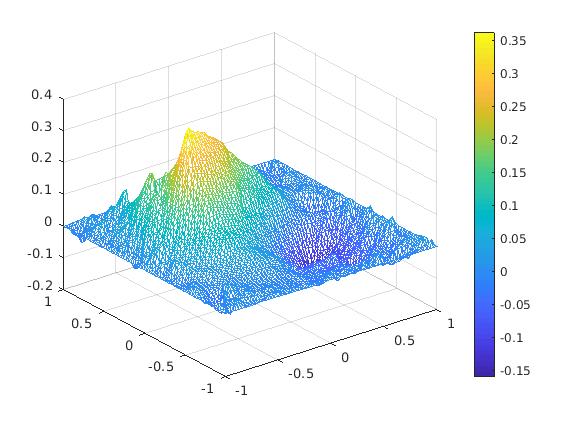}
\end{center}
\caption{With $\rho_\ell = h_\ell^2$: computed numerical solution $\beta_\ell$  at the final iteration, the differences $N^{h_\ell}_{j^\dag}(\beta^\dag) - N^{h_\ell}_{j_{\delta_\ell}} ({\beta_\ell})$ and $D^{h_\ell}_{\widehat{g}^\dag}(\beta^\dag) - D^{h_\ell}_{g^\dag_{\delta_\ell}} ({\beta_\ell})$.}
\label{h2}
\end{figure}

\begin{figure}[H]
\begin{center}
\includegraphics[scale=0.21]{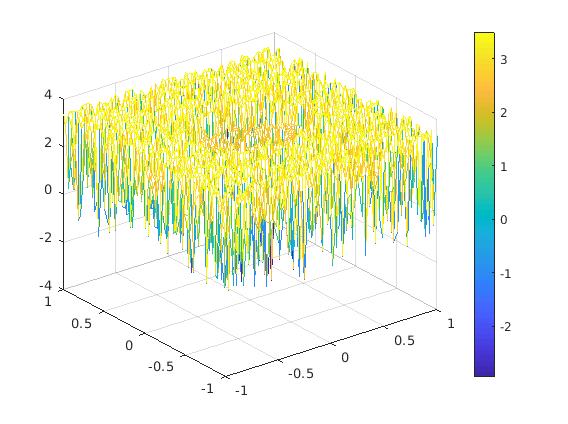}
\includegraphics[scale=0.21]{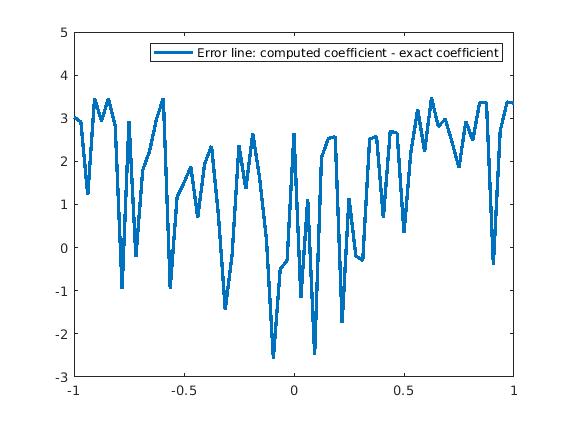}
\includegraphics[scale=0.21]{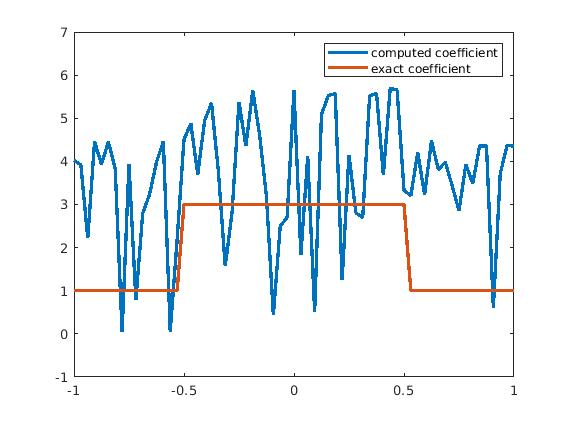}
\end{center}
\caption{With $\rho_\ell = h_\ell^2$: Differences $\beta_\ell - I^{h_\ell}_1 \beta^\dag$,~ $\beta_\ell(x_1,0) - I^{h_\ell}_1 \beta^\dag(x_1,0)$ and the graphs $\beta_\ell(x_1,0)$ (blue one) as well as $I^{h_\ell}_1 \beta^\dag(x_1,0)$.}
\label{h3}
\end{figure}

\Cref{h4} and \Cref{h5} perform analogous numerical computations for $\rho=0.1h_\ell^2$.

\begin{figure}[H]
\begin{center}
\includegraphics[scale=0.21]{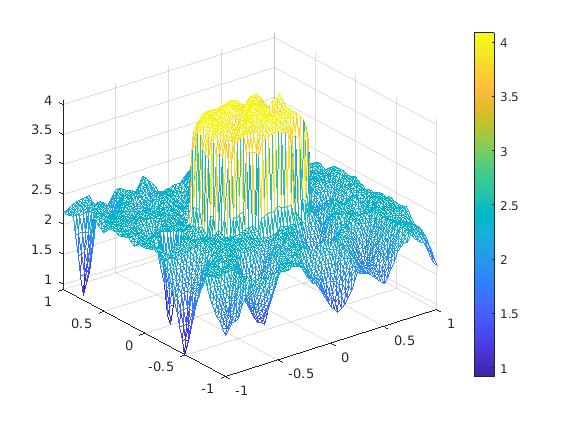}
\includegraphics[scale=0.21]{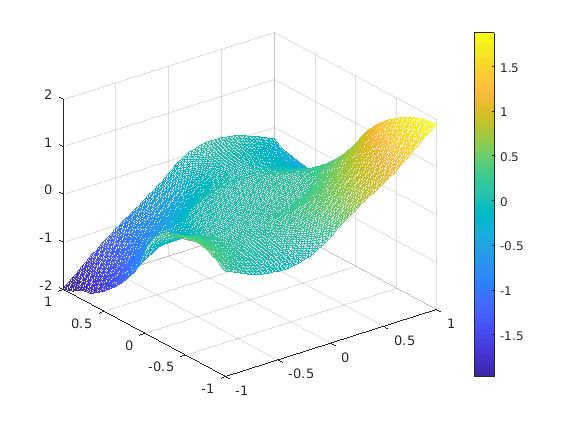}
\includegraphics[scale=0.21]{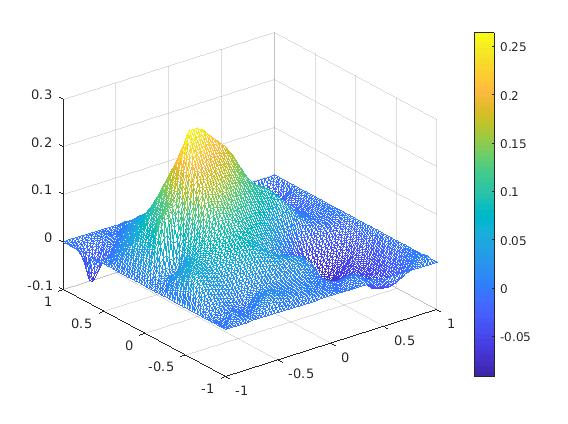}
\end{center}
\caption{With $\rho_\ell = 0.1 h_\ell^2$: computed numerical solution $\beta_\ell$  at the final iteration, the differences $N^{h_\ell}_{j^\dag}(\beta^\dag) - N^{h_\ell}_{j_{\delta_\ell}} ({\beta_\ell})$ and $D^{h_\ell}_{\widehat{g}^\dag}(\beta^\dag) - D^{h_\ell}_{g^\dag_{\delta_\ell}} ({\beta_\ell})$.}
\label{h4}
\end{figure}

\begin{figure}[H]
\begin{center}
\includegraphics[scale=0.21]{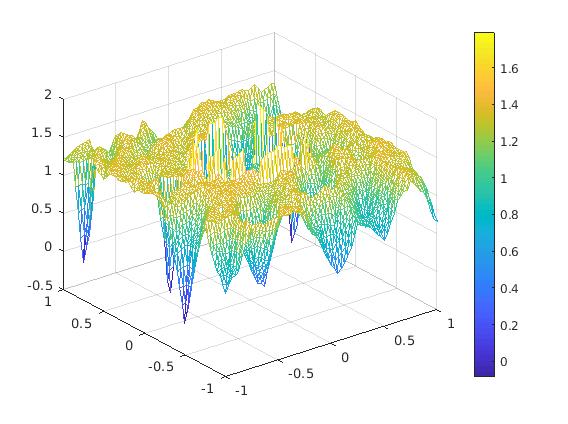}
\includegraphics[scale=0.21]{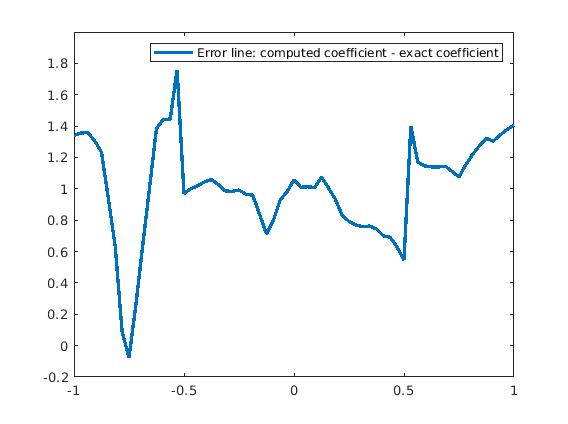}
\includegraphics[scale=0.21]{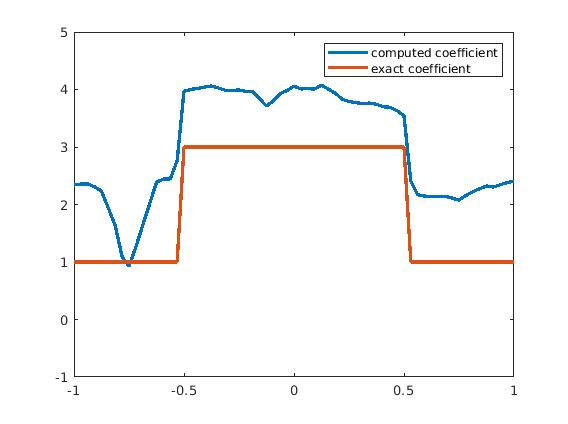}
\end{center}
\caption{With $\rho_\ell = 0.1h_\ell^2$: Differences $\beta_\ell - I^{h_\ell}_1 \beta^\dag$,~ $\beta_\ell(x_1,0) - I^{h_\ell}_1 \beta^\dag(x_1,0)$ and the graphs $\beta_\ell(x_1,0)$ (blue one) as well as $I^{h_\ell}_1 \beta^\dag(x_1,0)$.}
\label{h5}
\end{figure}
\end{exam}

%==============================
%==============================

\begin{exam}\label{exam3}
We now consider the case $\Gamma_{\mbox{observation}} := \Gamma_{\mbox{bottom-left}}$ includes the bottom surface and the left surface of the domain $\Omega$, i.e. $\Gamma_{\mbox{bottom-left}} = \{ x = (x_1,x_2) \in \overline{\Omega} ~|~ x_2=-1\} \cup \{ x = (x_1,x_2) \in \overline{\Omega} ~|~ x_1=-1\}$. In this case $\theta_\ell$ in \cref{3-7-17ct1} is changed to $\theta_\ell = \frac{1}{2} h_\ell \sqrt{10\cdot\rho_\ell}$, while the regularization parameter $\rho = \rho_\ell = 0.001 \sqrt{h_\ell}$ as in \Cref{exam1}. The computational result shows in \Cref{b3} and \Cref{b12} below.

\begin{table}[H]
\begin{center}
\begin{tabular}{|c|l|l|l|l|l|}
\hline \multicolumn{6}{|c|}{ { Error history for $\Gamma_{\mbox{observation}} := \Gamma_{\mbox{bottom-left}}$} 
\vspace{0.05cm}
}\\

\hline
$\ell$ &\scriptsize $\delta_\ell$  &\scriptsize {$L^2_\beta$} &\scriptsize {$L^2_N$} &\scriptsize {$L^2_M$} &\scriptsize {$L^2_D$}\\
\hline
4   & 7.6402e-2& 1.1379 & 0.3066& 0.2618 & 0.1911\\
\hline
8   & 3.2528e-2& 0.5102& 0.2215& 0.1834 & 0.1125\\
\hline
16  & 1.4637e-2& 0.2056& 0.1443& 0.1196 & 5.2079e-2\\
\hline
32  & 5.4159e-3& 0.1192& 5.6266e-2& 4.2652e-2 & 2.0737e-2\\
\hline
64  & 2.3331e-3& 6.6896e-2& 1.4888e-2& 1.0937e-2 &6.1714e-3\\
\hline
\end{tabular}
\end{center}
\caption{Refinement level $\ell$, measurement noise $\delta_\ell$, and errors $L^2_\beta$, $L^2_N$, $L^2_M$, $L^2_D$.}
\label{b3}
\end{table}

\begin{table}[H]
\begin{center}
\begin{tabular}{|c|c|c|c|c|}
 \hline \multicolumn{5}{|c|}{ { Experimental order of convergence \big($\Gamma_{\mbox{observation}} := \Gamma_{\mbox{bottom-left}}$\big)} }\\
 \hline
$\ell$ &\scriptsize { EOC$_{L^2_\beta}$} &\scriptsize { EOC$_{L^2_N}$} &\scriptsize { EOC$_{L^2_M}$} &\scriptsize { EOC$_{L^2_D}$}\\
\hline
4 & --& -- & -- &--\\
\hline
8 & 1.1572 & 0.4691 & 0.5135& 0.7644\\
\hline
16 & 1.3112 & 0.6182 & 0.6168 &1.1112 \\
\hline
32 & 0.7865 & 1.3587 & 1.4875 &1.3285\\
\hline
64 & 0.8334 & 1.9181 & 1.9634 &1.7485\\
\hline
Mean of EOC 
& 1.0221 & 1.0910 & 1.1453 & 1.2381 \\
\hline
\end{tabular}
\caption{Experimental order of convergence between finest and coarsest
level for $L^2_\beta$, $L^2_N$, $L^2_M$ and $L^2_D$.}
\label{b12}
\end{center}
\end{table}

Compared with \Cref{b1}, we do not see the difference clearly in the obtained numerical result between $\Gamma_{\mbox{observation}} := \Gamma_{\mbox{bottom}}$ and $\Gamma_{\mbox{observation}} := \Gamma_{\mbox{bottom-left}}$. \Cref{b11} and \Cref{b12} show that the experimental order of convergence for $L^2$-error of the sought coefficient is of the first order. To the best of our knowledge, there have been no papers providing an error bound for the identification problem mentioned in the present work so far. Nevertheless, we would like to distinguish that utilizing distributed observations for the diffusion coefficient identification problem, under the suitable source conditions authors of \cite{falk83,haqu14,kolo88} also proved the first order convergence rate $\mathcal{O}(h)$ of the piecewise linear finite element approximations to the identified solution in the $L^2$-norm. 

\Cref{h6} below performs the graphs of the computation in this case.

\begin{figure}[H]
\begin{center}
\includegraphics[scale=0.15]{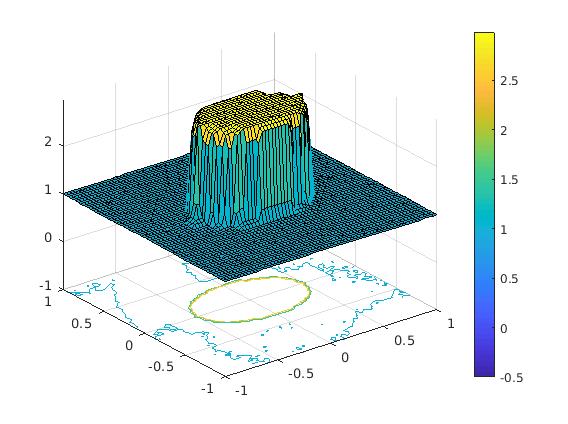}
\includegraphics[scale=0.15]{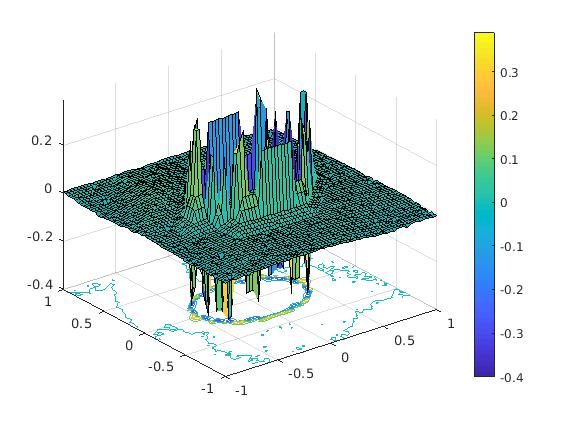} 
\includegraphics[scale=0.15]{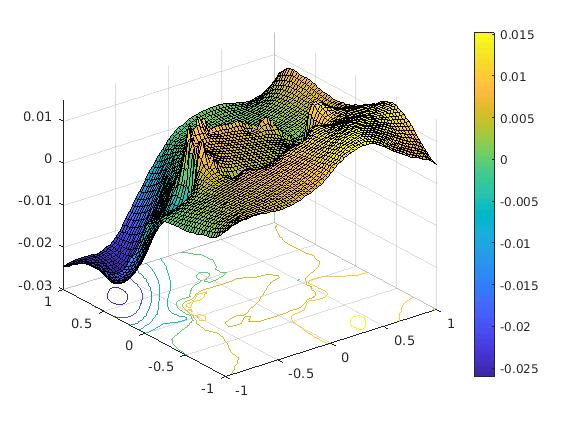}
\includegraphics[scale=0.15]{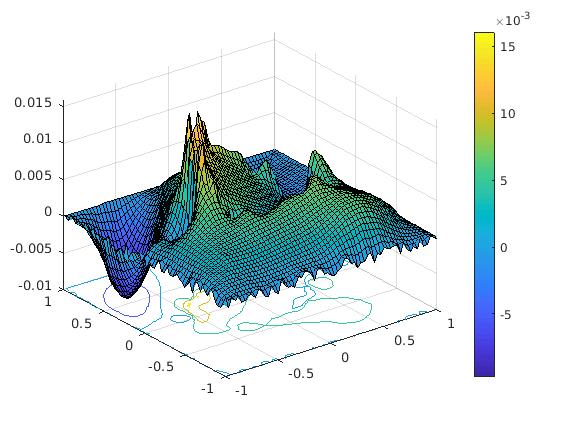}
\end{center}
\caption{Computed numerical solution $\beta_\ell$ of the algorithm at the final 556$^{\mbox{\tiny th}}$-iteration, and the differences $\beta_\ell - I^{h_\ell}_1 \beta^\dag$, $N^{h_\ell}_{j^\dag}(\beta^\dag) - N^{h_\ell}_{j_{\delta_\ell}} ({\beta_\ell})$, $D^{h_\ell}_{\widehat{g}^\dag}(\beta^\dag) - D^{h_\ell}_{g^\dag_{\delta_\ell}} ({\beta_\ell})$ for $\ell=64$, $\rho = \rho_\ell = 0.001 \sqrt{h_\ell}$ and $\delta_\ell=2.3331$e-3.}
\label{h6}
\end{figure}
\end{exam}

%================================
%================================

\begin{exam}\label{exam4}
In the last example we assume that $I$ multiple measurements on the bottom surface and the left surface $\left(j_\delta^i,g_\delta^i \right)_{i=1,\ldots,I}$ are available. Then, the cost functional $J^h_{\delta,\rho}$ and the problem $\big(\mathcal{P}^h_{\rho,\delta}\big)$ can be rewritten as 
$$
\min_{\beta\in\mathcal{S}_{ad}} \bar{J}^h_{\delta,\rho}(\beta), \quad  \bar{J}^h_{\delta,\rho}(\beta) := \frac{1}{I} \sum_{i=1}^I \big\|N^h_{j^i_\delta}(\beta) - M^h_{g^i_\delta}(\beta) \big\|^2_\Omega + \rho\|\beta - \beta^*\|^2_\Omega \eqno \big(\bar{\mathcal{P}}^h_{\delta,\rho}\big)
$$ 
which also attains a solution $\bar{\beta}^h_{\delta,\rho}$. Let $j^\dag_{(A,B,C,D)}$ be given by \cref{22-12-18ct1} which depends on the constants $A, B, C, D$ and $g^\dag_{(A,B,C,D)} :=  \gamma_\Gamma N_{j^\dag_{(A,B,C,D)}} (\beta^\dag)$. The noisy observations are assumed to give by
\begin{align*}
\left( j^{(A,B,C,D)}_{\delta_{\ell}}, g^{(A,B,C,D)}_{\delta_{\ell}} \right) = \left( j^\dag_{(A,B,C,D)}+ \theta\cdot R_{j^\dag_{(A,B,C,D)}}, g^\dag_{(A,B,C,D)}+ \theta\cdot R_{g^\dag_{(A,B,C,D)}}\right),
\end{align*}
where $\theta >0$ is independent of the mesh size, regularization parameter and noise level. $R_{j^\dag_{(A,B,C,D)}}$ and $R_{g^\dag_{(A,B,C,D)}}$ denote $\partial M^{h_\ell}\times 1$-matrices of random numbers on the interval $(-1,1)$, and $\partial M^{h_\ell}$ is the set of boundary nodes of the triangulation $\mathcal{T}^{h_\ell}$ which belong to $\Gamma_{\mbox{observation}} := \Gamma_{\mbox{bottom-left}}$.

The numerical result in \Cref{b4} presents for $\theta=0.2$ and with respect to 

\quad $\bullet$ $I=1$ measurement: $(A,B,C,D) = (1,-2,3,-4)$,

\quad $\bullet$ $I=6$ measurements: fixing $D=-4$ and taking $(A,B,C)$ equals to all permutations of the set $\{1,-2,3\}$,

\quad $\bullet$ $I=16$ measurements: taking $(A,B,C,D)$ equals to all permutations of the set $\{1,-2,3,-4\}$.

We observe that the use of multiple measurements improves the solution to yield an acceptable result even in the presence of relatively large noise as \Cref{b4} below.

\begin{table}[H]
\begin{center}
\begin{tabular}{|c|l|l|l|l|l|}
\hline \multicolumn{6}{|c|}{ {Numerical result for $\ell=64$, $\theta=0.2$} }\\
\hline
\mbox{Number of measurements} $I$  &\scriptsize {\bf Iterate}  &\scriptsize $L^2_\beta$ &\scriptsize $L^2_N$ &\scriptsize $L^2_M$ &\scriptsize $L^2_D$\\
\hline
1  & 600& 0.5889&  0.3040 & 0.2397 &0.1338 \\
\hline
6 & 600&  0.3846& 0.1597& 0.1215 & 6.2670e-2\\
\hline
16 & 600& 0.2746& 8.7720e-2& 7.1346e-2 & 4.5228e-2\\
\hline
\end{tabular}
\end{center}
\caption{Numerical result for $\ell=64$, $\rho = \rho_\ell = 0.001 \sqrt{h_\ell}$, $\theta=0.2$, i.e. $\delta_\ell=0.4773$, and with multiple measurements $I=1,6,16$.}
\label{b4}
\end{table}

Finally, in \Cref{h7}--\Cref{h9} we perform the graphs of the computation for the multiple measurements $I=1,6,16$, respectively, which include the computed numerical solution $\beta_\ell$, and the differences $\beta_\ell - I^{h_\ell}_1 \beta^\dag$, $N^{h_\ell}_{j^\dag}(\beta^\dag) - N^{h_\ell}_{j_{\delta_\ell}} ({\beta_\ell})$, $D^{h_\ell}_{\widehat{g}^\dag}(\beta^\dag) - D^{h_\ell}_{g^\dag_{\delta_\ell}} ({\beta_\ell})$ for $\ell=64$, $\rho = \rho_\ell = 0.001 \sqrt{h_\ell}$ and $\delta_\ell=0.4773$.

\begin{figure}[H]
\begin{center}
\includegraphics[scale=0.15]{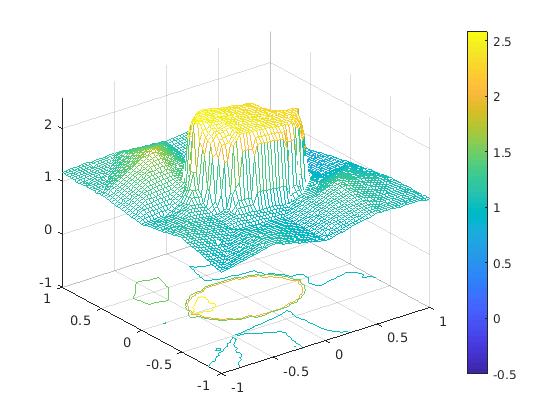}
\includegraphics[scale=0.15]{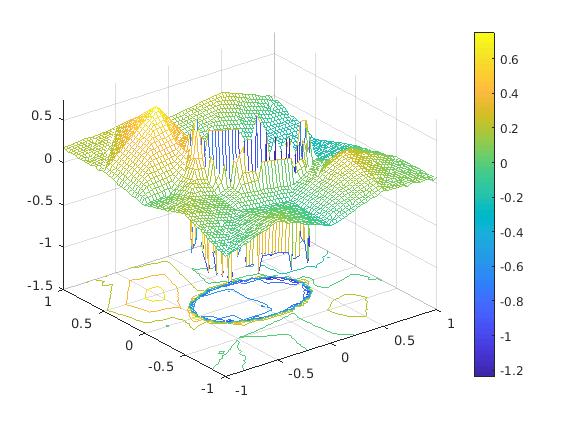} 
\includegraphics[scale=0.15]{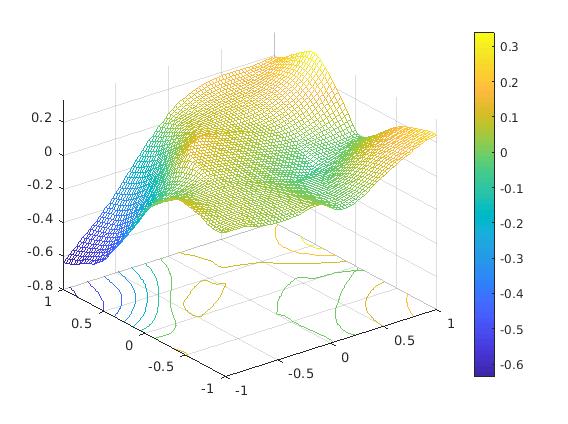}
\includegraphics[scale=0.15]{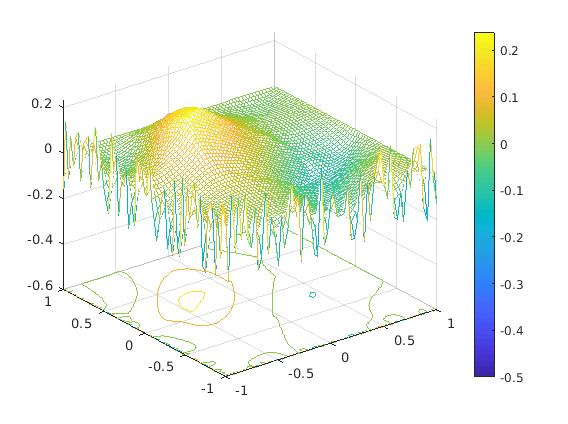}
\end{center}
\caption{$I=1$ measurement.}
\label{h7}
\end{figure}

\begin{figure}[H]
\begin{center}
\includegraphics[scale=0.15]{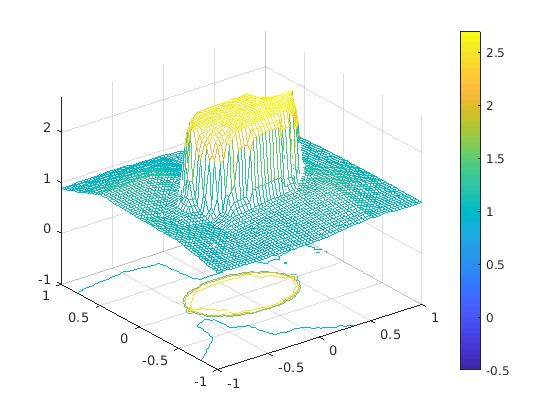}
\includegraphics[scale=0.15]{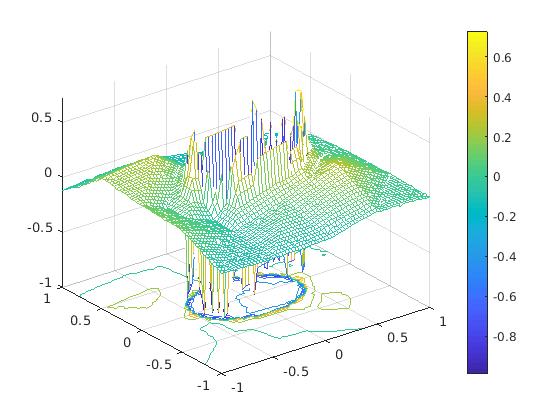} 
\includegraphics[scale=0.15]{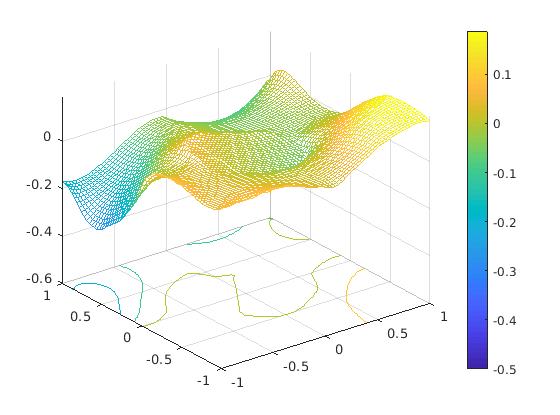}
\includegraphics[scale=0.15]{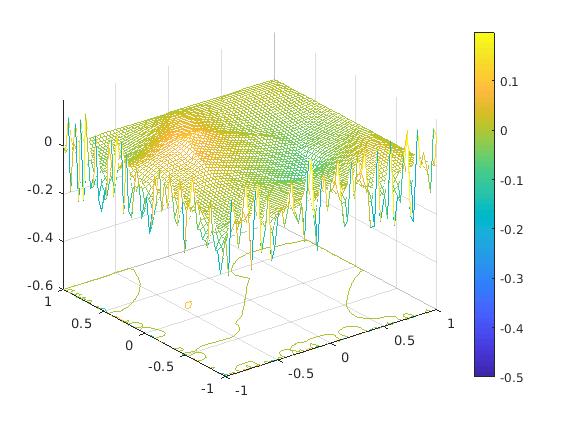}
\end{center}
\caption{$I=6$ measurements.}
\label{h8}
\end{figure}

\begin{figure}[H]
\begin{center}
\includegraphics[scale=0.15]{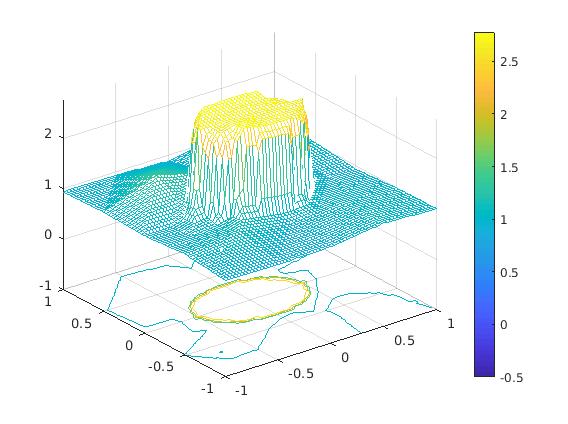}
\includegraphics[scale=0.15]{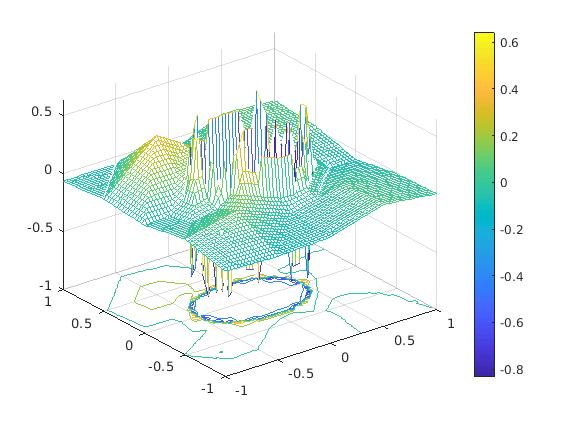} 
\includegraphics[scale=0.15]{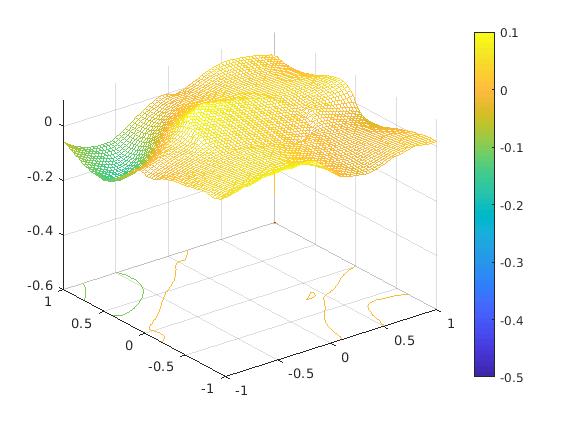}
\includegraphics[scale=0.15]{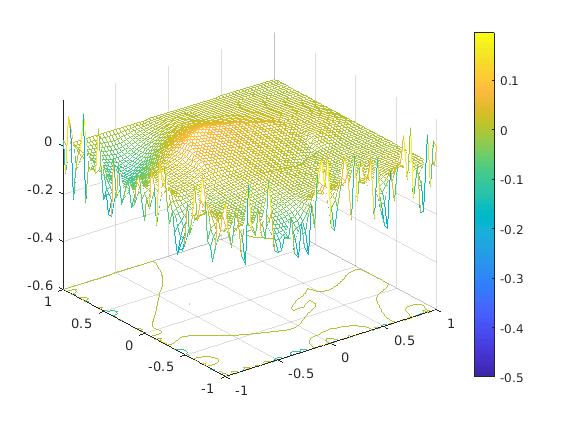}
\end{center}
\caption{$I=16$ measurements.}
\label{h9}
\end{figure}
\end{exam}

\section{Conclusions}\label{Conclusions}
The paper is devoted to the problem of identifying the reaction coefficient $\beta = \beta(x)$ in the equation $-\nabla \cdot \big(\boldsymbol{\alpha}(x) \nabla \Phi(x) \big) + \beta(x) \Phi(x) = f(x), ~ x\in \Omega$ supplemented with the boundary conditions $\boldsymbol{\alpha}(x) \nabla \Phi(x) \cdot \vec{\boldsymbol{n}}(x) +\sigma(x) \Phi(x) = j^\dag(x)$,~ $x\in \Gamma$,~ $\boldsymbol{\alpha}(x) \nabla \Phi(x) \cdot \vec{\boldsymbol{n}}(x) +\sigma(x) \Phi(x) = j_0(x), ~ x\in \partial\Omega\setminus\Gamma$  and $\Phi(x) = g^\dag(x),~ x\in \Gamma$ from the measurement data $\big(j_{\delta},g_{\delta}\big)$ of the exact $\big(j^\dag,g^\dag\big)$. In this context, the special functions $\boldsymbol{\alpha}, f$ and $\sigma$ are given, where $\vec{\boldsymbol{n}}$ is the unit outward normal on $\partial\Omega$.

With the available measurement data $\big(j_{\delta},g_{\delta}\big)$ at hand, we for each $\beta\in\mathcal{S}_{ad}$ consider simultaneously two problems
\begin{align*}
&-\nabla \cdot \big(\boldsymbol{\alpha} \nabla u \big) + \beta u = f  \mbox{~in~}  \Omega,~
\boldsymbol{\alpha} \nabla u \cdot \vec{\boldsymbol{n}} +\sigma u = \begin{cases}
j_\delta  \mbox{~on~} \Gamma, \\
j_0 \mbox{~on~} \partial\Omega\setminus\Gamma
\end{cases}\\
&-\nabla \cdot \big(\boldsymbol{\alpha} \nabla v \big) + \beta v = f \mbox{~in~} \Omega, 
v = g_\delta \mbox{~on~} \Gamma,~
\boldsymbol{\alpha} \nabla v \cdot \vec{\boldsymbol{n}} +\sigma v = j_0 \mbox{~on~} \partial\Omega\setminus\Gamma
\end{align*}
and denote respectively by $N_{j_\delta}(\beta)$ and $M_{g_\delta}(\beta)$ their unique weak solutions. A minimizer $\beta_{\delta,\rho}$ of the Tikhonov regularized minimization problem
$$
\min_{\beta \in \mathcal{S}_{ad}} J_{\delta,\rho}(\beta), \quad J_{\delta,\rho} (\beta) := \big\|N_{j_\delta}(\beta)- M_{g_\delta}(\beta)\big\|^2_{L^2(\Omega)} + \rho R(\beta,\beta^*) \eqno \left(\mathcal{P}_{\delta,\rho}\right) 
$$
is considered as reconstruction, where $\rho>0$ is the regularization parameter and $\beta^*$ is an a priori estimate of the true coefficient, with the regularization term $R(\beta,\beta^*) := \|\beta-\beta^*\|^2_{L^2(\Omega)}$.
Let $N^h_{j_\delta}(\beta)$ and $M^h_{g_\delta}(\beta)$ be corresponding approximations of $N_{j_\delta}(\beta)$ and $M_{g_\delta}(\beta)$ in the finite dimensional space $\mathcal{V}^h_1$ of piecewise linear, continuous finite elements. We then examine the discrete regularized problem
$$
\min_{\beta\in\mathcal{S}_{ad}} J^h_{\delta,\rho}(\beta), \quad  J^h_{\delta,\rho}(\beta) := \big\|N^h_{j_\delta}(\beta) - M^h_{g_\delta}(\beta) \big\|^2_{L^2(\Omega)} + \rho R(\beta,\beta^*) \eqno \big(\mathcal{P}^h_{\delta,\rho}\big)
$$ 
which also attains a minimizer $\beta^h_{\delta,\rho}$.

We show that when $\delta$ and $\rho$ are fixed the sequence of minimizers $\big(\beta^h_{\delta,\rho}\big)_{h>0}$ to $\big(\mathcal{P}^h_{\delta,\rho}\big)$ can be extracted a subsequence which converges in the $L^2(\Omega)$-norm to a solution $\beta_{\delta,\rho}$ of $\big(\mathcal{P}_{\delta,\rho}\big)$ as the mesh size $h\to 0$. Furthermore as $h,\delta \to 0$ and with an appropriate a priori regularization parameter choice $\rho=\rho(h,\delta) \to 0$,
the whole sequence $\big(\beta^h_{\delta,\rho}\big)_{\rho>0}$ converges in the $L^2(\Omega)$-norm to the $\beta^*$-minimum-$R$ solution $\beta^\dag$ of the
identification problem.
The corresponding state sequences $\big( N^{h}_{j_{\delta}}\big(\beta^h_{\delta,\rho}\big)\big)_{\rho>0}$ and $\big( M^{h}_{g_{\delta}}\big(\beta^h_{\delta,\rho}\big)\big)_{\rho>0}$ then converge in the $H^1(\Omega)$-norm to the exact state $\Phi^\dag = \Phi(j^\dag,g^\dag,\beta^\dag)$ of the problem.

For the particular interest in estimating probably discontinuous reaction coefficients  one can employ the total variation regularization $R(\beta) = \int_\Omega |\nabla \beta|$ which was originally introduced in image denoising by authors of \cite{ROF92}
and was also applied to several ill-posed and inverse problems, see, e.g., \cite{AcaVo,aubert, CasKuPo, ChavenKu}. 

Alternatively, starting with \cite{Dau04}, the sparsity regularization has been applied to the parameter identification problem for PDEs, see the review paper \cite{JiMa12}. In this situation the penalty term is defined by $R(\beta) = \sum_i w_i|\langle \beta,\varphi_i\rangle|^p$ for $1\le p\le 2$ and with weights $w_i\ge \underline{w} >0$. Here $\{\varphi_i\}_i$ is an orthonormal basis or overcomplete frame of the space $L^2(\Omega)$. In practice the system $\{\varphi_i\}_i$ is chosen to be highly dependent on the structure of the reconstructed solution, that could be a Fourier representation for oscillatory features, wavelets for pointwise singularities \cite{stu12}.
Together with the total variation penalty term, adopting the sparsity regularization method for the identification problem may be a work for us in future.

\section*{Acknowledgments}
The author would like to thank the Editor and Referees for their valuable comments and suggestions.

\end{document}